\documentclass[12pt]{amsart}
\usepackage[margin=2.5cm]{geometry}
\usepackage[T1]{fontenc}
\usepackage{microtype}
\usepackage{amssymb}
\usepackage{float}
\usepackage{tabularx}
\usepackage{booktabs}
\usepackage{hhline,multirow} 
\usepackage[matrix,arrow,curve]{xy}
\usepackage{hyperref}

\usepackage{longtable}
\usepackage{mathrsfs, amssymb, array, upgreek, float}
\usepackage{tikz}
\usetikzlibrary{shapes.misc}
\usetikzlibrary{decorations.pathreplacing}

\def\sizec{0.2em}

\tikzset{cross/.style={cross out, draw=black, minimum size=\sizec, inner sep=0pt, outer sep=0pt},
cross/.default={1pt}}

\usepackage[normalem]{ulem}

\newcommand{\KK}{\mathbb{K}}

\newcommand{\ZZ}{\mathbb{Z}}
\newcommand{\PP}{\mathbb{P}}
\newcommand{\QQ}{\mathbb{Q}}

\newcommand{\Ga}{\mathbb{G}_{\mathrm{a}}}
\newcommand{\Gm}{\mathbb{G}_{\mathrm{m}}}
\newcommand{\Aff}{\mathbb{A}}
\newcommand{\Pic}{\operatorname{Pic}}
\newcommand{\bF}{\mathbf{F}}
\newcommand{\bx}{\mathbf{x}}
\newcommand{\chit}{\upchi_{\mathrm{top}}}

\newcommand{\Di}{\mathfrak{D}}

\newcommand{\jj}{\operatorname{j}}

\newcommand{\rr}{\operatorname{r}}
\newcommand{\hh}{\mathrm{h}^{1,2}}
\newcommand{\rk}{\operatorname{rk}}
\newcommand{\Cl}{\operatorname{Cl}}

\newcommand{\OOO}{{\mathscr{O}}} 

\newcommand{\PPP}{{\mathscr{P}}}

\newcommand{\LLL}{{\mathscr{L}}} 
\newcommand{\NNN}{{\mathscr{N}}}

\newcommand{\TTT}{{\mathscr{T}}}
\newcommand{\wt}{\operatorname{wt}}

\newcommand{\Sing}{\operatorname{Sing}}
\newcommand{\Supp}{\operatorname{Supp}}

\newcommand{\Aut}{\operatorname{Aut}}
\newcommand{\red}{\operatorname{red}}
\newcommand{\Bs}{\operatorname{Bs}}

\newcommand{\PSL}{\operatorname{PSL}}

\newcommand{\SL}{\operatorname{SL}}
\newcommand{\GL}{\operatorname{GL}}
\newcommand{\PGL}{\operatorname{PGL}}
\newcommand{\Proj}{\operatorname{Proj}}

\newcommand{\cc}{\mathrm{c}}
\newcommand{\g}{\mathrm{g}}

\newcommand{\p}{\mathrm{p}_{\mathrm{a}}}

\newcommand{\type}[2]{$\mathrm{#1}_{#2}$}

\newcommand{\mult}{\operatorname{mult}}
\newcommand{\ord}{\operatorname{ord}}

\newcommand{\xref}[1]{\textup{\ref{#1}}}

\renewcommand\labelenumi{\rm (\roman{enumi})}
\renewcommand\theenumi{\rm (\roman{enumi})}

\makeatletter
\@addtoreset{equation}{section}
\makeatother

\theoremstyle{plain}
\newtheorem{claim}[subsection]{Claim}
\newtheorem{lemma}[subsection]{Lemma}
\newtheorem{proposition}[subsection]{Proposition}
\newtheorem{theorem}[subsection]{Theorem}
\newtheorem{corollary}[subsection]{Corollary}

\theoremstyle{definition}
\newtheorem{definition}[subsection]{Definition}
\newtheorem{example}[subsection]{Example}

\newtheorem{remark}[subsection]{Remark}
\newtheorem{notation}[subsection]{Notation}
\newtheorem*{notation*}{Notation}
\newtheorem{construction}[subsection]{Construction}
\newtheorem{setup}[subsection]{Setup}

\newcounter{NN}
\renewcommand{\theNN}{\rm\arabic{NN}${}^o$}
\def\nr{\refstepcounter{NN}{\theNN}}

\title{Double Veronese cones with singularities}
\author{Yuri Prokhorov}

\thanks{
The author was partially supported by the HSE University Basic Research Program. }
\address{
Steklov Mathematical Institute of Russian Academy of Sciences, Moscow, Russia
\newline\indent
Department
of Algebra, Moscow State
University, Russia
\newline\indent
National Research University Higher School of Economics, Moscow, Russia
}
\email{prokhoro@mi.ras.ru}

\begin{document}

\begin{abstract}
We study double Veronese cones -- three-dimensional del Pezzo varieties of degree one -- with terminal Gorenstein singularities. We prove sharp bounds for the number of nodes, determine the 
structure of the automorphism group, and establish criteria for rationality and unirationality. 
In particular, we exhibit a $\mathbb{Q}$-factorial nodal double Veronese cone
with $21$ nodes.
\end{abstract}
\maketitle

\section{Introduction}
\label{sec:intro}

A \textit{double Veronese cone} is a three-dimensional del Pezzo variety of degree~$1$, i.e. a projective threefold~$X$ whose anticanonical class has a decomposition
$-K_X=2A_X$, where $A_X$ is an ample Cartier divisor
such that $A_X^3=1$.
These varieties can be described as double covers of the cone $Y\subset\PP^6$ over the Veronese surface branched over the section of $Y$ by a cubic hypersurface.
Equivalently, such a variety admits an embedding into the weighted projective space $\PP(1^3,2,3)$ as a hypersurface of degree~$6$. 
The geometry of such threefolds is intimately related to the elliptic fibration 
obtained by blowing up the base point of the half-anticanonical linear system $|A_X|$.

In this paper we systematically study double Veronese cones with terminal Gorenstein singularities.
Our main focus lies on several complementary aspects:

\subsubsection*{Singularities and $\QQ$-factoriality.}
We prove that a $\QQ$-factorial nodal double Veronese cone can have at most $21$ nodes, and that this bound is attained by some example (Theorem~\ref{thm:Q-fact}). Conversely, a
non‑$\mathbb{Q}$-factorial nodal double Veronese cone must have at least $12$ nodes, and the case of exactly $12$ nodes occurs.

\subsubsection*{Automorphism group.}
We show that the identity component of the automorphism group $\Aut(X)$ of a double Veronese cone with
infinite $\Aut(X)$ is  either a torus $\Gm$ or the additive group $\Ga$; explicit equations
for the corresponding actions are provided (Propositions~\ref{prop:torus} and~\ref{prop:Ga}).
Moreover,  for nodal double Veronese cones the automorphism group is finite.

\subsubsection*{Rationality and unirationality.}
We establish that every singular double Veronese cone is unirational. Furthermore, we give some sufficient conditions for rationality (Theorem~\ref{thm:rat-unirat}). 

The paper is organized as follows. Section~\ref{sec:prelim} collects notation and basic facts about double Veronese cones. In Section~\ref{sec:sing} we analyze the singularities and their relation to 
the certain plane curves. Section~\ref{sec:lines} is devoted to the geometry of lines and their families, which is essential for the study of certain birational contractions. 
Section~\ref{sec:class} discusses the class group and $\QQ$-factoriality; here we prove the sharp bounds on the number of nodes (Theorem~\ref{thm:Q-fact}). The automorphism group is
investigated in Section~\ref{sec:aut}, where we classify actions of $\Gm$ and $\Ga$ and prove finiteness for nodal varieties. In Section~\ref{sec:rat} we prove the rationality and unirationality 
criteria. An appendix (Section~\ref{sec:app}) 
collects auxiliary results on Du Val del Pezzo surfaces over algebraically  nonclosed fields, which are needed for the rationality proofs.

\textbf{Acknowledgements.}
I would like to thank Alexander Kuznetsov, Konstantin Loginov, Constantin Shramov,
and Guolei Zhong for useful conversations, comments, and suggestions.
I also thank the referee for careful reading the manuscript and numerous comments.

\section{Preliminaries}
\label{sec:prelim}

Throughout this paper, if not stated otherwise, we work over an algebraically closed field $\Bbbk$ 
of characteristic $0$. 
All varieties are supposed to be irreducible and reduced.
We use the following standard notation.

\subsection*{Notation}
\begin{itemize}

\item
$\Cl(X)$ is the Weil divisor class group of a normal variety~$X$;
\item
$\rr(X):=\rk\Cl(X)$;
\item
$\Gm$ and $\Ga$ are the multiplicative and additive one‑parameter groups.

\end{itemize}

A normal projective variety $X$ of dimension $n = \dim X \ge 2$ with at worst terminal Gorenstein singularities is called \textit{del Pezzo} (resp.
\textit{weak del Pezzo}) if its anticanonical class has the following representation
\begin{equation*}
-K_X=(n-1)A_X,
\end{equation*}
where $A_X$ is an ample Cartier divisor
(resp. a nef and big Cartier divisor).
If $X$ is a weak del Pezzo variety, then by the Base Point Free Theorem the linear system $|mA_X|$
for $m\gg 0$ defines a morphism $\Phi_{|mA_X|}: X\to \PP^N$ (see e.g. \cite[Theorem~3.3]{KM:book}).
In this situation we say that $X$ is an \textit{almost del Pezzo variety}, if 
$\Phi_{|mA_X|}$ 
does not contract any divisor.
Note that the Picard group $\Pic(X)$ of a weak del Pezzo variety $X$ is torsion free
and finitely generated \cite[Proposition~2.1.2]{IP99}, so the class $A_X$ is well-defined. 
It is called the \textit{fundamental} class. The top self-intersection number
$(A_X)^{\dim X}$ is called the degree of a weak del Pezzo variety $X$.

If~$X$ is a del Pezzo variety and $\xi: \hat{X}\to X$ is a $\QQ$-factorialization, then $\hat{X}$ is almost del Pezzo.
Conversely, if $Y$ is an almost del Pezzo variety and 
\begin{equation*}
Y_{\mathrm{can}}:=\operatorname{Proj} \bigoplus_{m\ge 0} H^0(Y,\OOO_Y(-mK_Y))
\end{equation*}
its anticanonical model, then $Y_{\mathrm{can}}$ is a del Pezzo variety (see \cite[\S 2.2]{KP:dP}).
In this paper we discuss only the three-dimensional case.

\begin{definition}
A \textit{double Veronese cone} is a del Pezzo threefold of degree~$1$.
\end{definition}

For the following two assertions we refer to \cite{Shin1989} (see also \cite{Isk:anti-e, Fujita:book, P:Fano25e}).

\begin{theorem}
\label{thm:preV1}
Let~$X$ be a double Veronese cone and let $A_X:=-\frac12K_X$. 
\begin{enumerate}

\item
A general member of the linear system $|A_X|$ is a smooth del Pezzo surface of degree~$1$.

\item
The linear system $|A_X|$ has no fixed components, it has exactly one base point $O\in X$, which is a smooth point of~$X$, and defines a rational map $\tau: X \dashrightarrow B=\PP^2$.
Any member of $|A_X|$ is smooth at $O$.

\item
The linear system $|2A_X|=|-K_X|$ is base point free and defines a morphism 
\begin{equation*}
\pi: X\longrightarrow Y\subset \PP^6, 
\end{equation*}
where $Y=\pi(X)\simeq \PP(1^3,2)$ is the cone over the Veronese surface with vertex $\pi(O)$; 
the morphism~$\pi$ is finite of degree~$2$ and its branch divisor $R_Y$ is cut out on $Y$ by a cubic that does not pass through $\pi(O)$.

\item
The divisor $3A_X$ is very ample.

\item
The map $\tau$ and the morphism~$\pi$ are included in the commutative diagram
\begin{equation}
\label{eq:preV1}
\vcenter{
\xymatrix {
&\tilde X\ar[dl]_\sigma\ar [dr]!<-1.8em,0.1em> \ar[dd]^(.33){\varphi}&
\\
X\ar@{-->}[dr]^{\tau}\ar[rr]|!{[d];[r]}\hole^(.33)\pi & &Y\simeq \PP(1^3,2)\ar@{-->}[] !<0.7em,3.em>; [dl]_(.7){\delta}
\\
&B\simeq\PP^2&
} }
\end{equation} 
where $\sigma$ is the blow-up of $O$, $\varphi$ is an elliptic fibration, $\delta$ is the projection from the vertex of the cone, and the exceptional divisor $E:=\sigma^{-1}(O)$ is a section of 
$\varphi$.
\end{enumerate}
\end{theorem}
Below we identify $B$ with $\PP^2$.
By abuse of terminology we often say that $\tau: X\dashrightarrow B=\PP^2$ is 
an \textit{elliptic fibration}.
The notation in the diagram~\eqref{eq:preV1} will be fixed throughout this paper.

\begin{theorem}
Let~$X$ be a double Veronese cone. Then there is a natural embedding 
\begin{equation}
\label{eq:emb}
X \hookrightarrow \PP(1^3,2,3),
\end{equation} 
so that the image is a hypersurface of degree~$6$.
Conversely, if $X\subset \PP(1^3,2,3)$ is a hypersurface of degree~$6$ such that the singularities of $X$ are terminal 
and $X$ does not pass through the singular \textup(orbifold\textup) points of $\PP(1^3,2,3)$,
then $X$ is a double Veronese cone.

Under the embedding \eqref{eq:emb} the morphism~$\pi$ coincides with the projection $X\to \PP(1^3,2)=Y$
and the map $\tau$ coincides with the projection $X \dashrightarrow\PP(1^3)=B$.
For a suitable choice of coordinates $x_1,x_2,x_3,y,z$ in $\PP(1^3,2,3)$ the equation of~$X$ has the form
\begin{equation}
\label{eq:eq}
z^2+y^3+y\phi_4(x_1,x_2,x_3)+\phi_6(x_1,x_2,x_3)=0,
\end{equation} 
where $\phi_4$ and $\phi_6$ are homogeneous polynomials of degree~$4$ and $6$, respectively, and $\phi_6\neq 0$.
\end{theorem}

The natural Galois involution of the double cover $\pi: X\to Y$ is called the \textit{Bertini involution}.
We denote it by 
\begin{equation*}
\upbeta: X\to X.
\end{equation*}
The fixed point locus of $\upbeta$ is the union of the ramification divisor $R_X$ of~$\pi$ and the isolated point $O$.
In the presentation \eqref{eq:eq} the involution $\upbeta$ acts via 
\begin{equation}
\label{eq:eq:Bertini}
\upbeta:(x_1,x_2,x_3,y,z) \longmapsto (x_1,x_2,x_3,y,-z).
\end{equation}

The absolute invariant of the fibers of $\tau$ defines an $\Aut(X)$-equivariant rational map 
\begin{equation*}
\jj: B \dashrightarrow \PP^1. 
\end{equation*}
In the presentation \eqref{eq:eq} it is given by
\begin{equation}
\label{eq:j}
\jj(Q)=\frac{4^4\cdot 27\phi_4(Q)^3}{4\phi_4(Q)^3+27\phi_6(Q)^2}.
\end{equation}

We need a description of singularities of members of the linear system $|A_X|$.
\begin{proposition}
\label{prop:hyp-sect}
Let~$X$ be a double Veronese cone and let $S\in |A_X|$ be an arbitrary member.
Then one of the following holds:
\begin{enumerate}

\item 
\label{prop:hyp-sect-a} 
$S$ is a del Pezzo surface with at worst Du Val singularities;

\item 
\label{prop:hyp-sect-b} 
$S$ is a generalized cone over an elliptic curve, i.e.~$S$ is obtained by contracting the negative section of the $\PP^1$-bundle $\PP(\OOO_C\oplus\LLL)$ over an elliptic curve 
$C$, where $\LLL$ is an ample line bundle on~$C$ with $\deg \LLL=1$; 

\item 
\label{prop:hyp-sect-c}
$S$ is a non-normal del Pezzo surface; if $\nu : S'\to S$ is its normalization, then $S'\simeq\PP^2$, $\nu^*\OOO_S(A_X)=\OOO_{\PP^2}(1)$, and the inverse image of the 
non-normal locus is a conic.
\end{enumerate}
The case \ref{prop:hyp-sect-c} does not occur if~$X$ is smooth.
\end{proposition}

\begin{proof}
By adjunction, the anticanonical divisor $-K_S$ is Cartier and ample. Moreover,
$K_S^2=1$. Hence~$S$ is a Gorenstein del Pezzo surface of degree~$1$.
If~$S$ is normal, then by \cite[Theorems~2.2 and~3.4]{Hidaka-Watanabe} we obtain~\ref{prop:hyp-sect-a} and~\ref{prop:hyp-sect-b}.
In the non-normal case, according to \cite[Theorem~1.1]{Reid:dP94} and \cite[Theorem~1.5]{Abe2003}, we have~\ref{prop:hyp-sect-c}.
The last assertion follows from \cite{Furushima-Tada:nnDP}. 
\end{proof}

We give two examples to show that both possibilities~\ref{prop:hyp-sect-b} 
and~\ref{prop:hyp-sect-c} of Proposition~\ref{prop:hyp-sect} occur.

\begin{example}
Let $X\subset \PP(1^3,2,3)$ be given by the equation
\begin{equation}
\label{eq:E8}
z^2+y^3+x_3(x_1^5+x_2^5)=0
\end{equation} 
and let $S:=X\cap \{x_3=0\}$. Then~$X$ is a double Veronese cone and~$S$ is a non-normal member of $|A_X|$ whose singular locus is the curve $\{x_3=y=z=0\}$.
\end{example}

\begin{example}
Let $X\subset \PP(1^3,2,3)$ be given by the equation 
\begin{equation*}
z^2+y^3+x_1^6+x_2^5x_3+x_3^6=0
\end{equation*}
and let $S:=X\cap \{x_3=0\}$. Then~$X$ is a smooth double Veronese cone and~$S$ is a member of $|A_X|$ that is a generalized cone.
\end{example}

\section{Singularities}
\label{sec:sing}
Recall that a three-dimensional singularity $V\ni P$ is said to be of type \type{cA}{n} (resp. \type{cD}{n}, \type{cE}{n}) if a general hyperplane section $H\subset V$ passing through $P$ is Du Val 
of type \type{A}{n} (resp. \type{D}{n}, \type{E}{n}).
Such a singularity is terminal whenever it is isolated. 
Conversely, any three-dimensional terminal Gorenstein singularity 
is isolated and has type \type{cA}{n}, \type{cD}{n}, or \type{cE}{n} \cite[Theorem~1.1]{Reid:MM}.

The notation introduced below will be kept throughout the whole paper.
\subsection{Associated curves}
Let $X\subset \PP(1^3,2,3)$ be a hypersurface given by the equation \eqref{eq:eq}
(at the moment, we do not assume that the singularities of~$X$ are terminal).
Define the following subschemes in $\PP^2_{x_1,x_2,x_3}$:
\begin{eqnarray*}
\Lambda_4 &:=& \Proj\, \Bbbk[x_1,x_2,x_3]/ (\phi_4),
\\
\Lambda_6 &:=& \Proj\, \Bbbk[x_1,x_2,x_3]/ (\phi_6),
\\
\Di &:=& \Proj\, \Bbbk[x_1,x_2,x_3]/ (4\phi_4^3+27\phi_6^2).
\end{eqnarray*}
Each of them is either $\PP^2$ or a purely one-dimensional plane curve set-theoretically given by
\begin{equation}
\label{eq:LambdaLambdaDi}
\begin{array}{rcl}
(\Lambda_4)_{\red} &=& \{Q\in \PP^2 \mid \jj(Q)=0\},
\\
(\Lambda_6)_{\red} &=& \{Q\in \PP^2 \mid \jj(Q)=1728 \},
\\
\Di_{\red} &=& \{Q\in \PP^2 \mid 4\phi_4(Q)^3+27\phi_6(Q)^2=0\}.
\end{array} 
\end{equation} 
In general, $\Lambda_4$, $\Lambda_4$, and $\Di$
can be reducible or non-reduced. 

\begin{lemma}
\label{lemma:sing0}
Suppose that there exists a non-zero homogeneous polynomial $\psi\in \Bbbk[x_1,x_2,x_3]$ of 
such that $\psi$ divides $\phi_4$ and
$\psi^2$ divides $\phi_6$.
Then the singularities of~$X$ are not isolated.
\end{lemma}

\begin{proof}
In this case~$X$ is singular along the curve $\{z=y=\psi=0\}$.
\end{proof}

\begin{corollary}
\label{cor:sing:comp}
Assume that the singularities of~$X$ are isolated. Then the following assertions hold:
\begin{enumerate}
\item 
$\Lambda_6\neq \PP^2$.
\item 
\label{cor:sing:comp:=}
If $\Lambda_4=\PP^2$, then $\Lambda_6$ has no multiple components.
\item 
\label{cor:sing:comp:-=}
If $\Lambda_4\neq \PP^2$, then 
any common component of $\Lambda_4$ and $\Lambda_6$ is not a 
multiple component of~$\Lambda_6$.
\end{enumerate}
\end{corollary}
In what follows we assume that $\phi_6\neq 0$ and $\phi_4$ is coprime with
$\phi_6/\gcd(\phi_4,\, \phi_6)$.
Then $\Lambda_6\neq \PP^2$ and $\Di\neq \PP^2$.
The curve $\Di$ is called the \textit{discriminant} of the elliptic fibration 
$\varphi: \tilde X\to \PP^2$
obtained by resolving the base point of the projection $\tau: X\dashrightarrow \PP^2$.
Thus with any double Veronese cone~$X$ we can associate three 
subschemes $\Lambda_4, \Lambda_6, \Di\subset \PP^2$ which, 
except for the case $\Lambda_4=\PP^2$, are 
(possibly reducible, non-reduced) plane curves.

\begin{remark}
\label{rem:fiber}
Any fiber $C:=\tau^{-1}(Q)$ over $Q\in B$ is a reduced irreducible curve of arithmetic genus~$1$
that is singular if and only if $Q\in \Di$.
Moreover, one of the following holds:
\begin{enumerate}
\item
$C$ is smooth if and only if $Q\in B\setminus \Di$;
\item
$C$ is a nodal rational curve if and only if 
$Q\in \Di \setminus (\Lambda_4\cap \Lambda_6)$;
\item
$C$ is a cuspidal rational curve if and only if
$Q\in \Lambda_4\cap \Lambda_6$.
\end{enumerate}
\end{remark}

Let~$X$ be a double Veronese cone.
Recall that~$X$ must be smooth at the point $O=\Bs |A_X|$, hence the map $\tau: X\dashrightarrow B$ is defined in a neighborhood of any singular point $P\in X$.
Denote 
\begin{equation*}
\Sigma:=\tau\big(\Sing(X)\big).
\end{equation*}
Since any fiber of $\tau$ has at most one singular point, the restriction $\tau_\Sigma: \Sigma \to \Sing(X)$ is a bijection.

\begin{lemma}
\label{lemma:singularities}
Let~$X$ be a double Veronese cone. Then 
\begin{eqnarray}
\label{eq:singularities:n}
\Sigma\setminus \Lambda_6&=&\Sing(\Di)\setminus \Lambda_6,
\\
\label{eq:singularities:3} 
\Sigma\cap \Lambda_6&=&\Lambda_4\cap \Sing(\Lambda_6).
\end{eqnarray}
\end{lemma}

\begin{proof}
Take a point $Q\in B=\PP^2$. The fiber $\tau^{-1}(Q)$ is a the scheme intersection of two members of $|A_X|$.
Hence the smoothness of $\tau^{-1}(Q)$ implies the smoothness of~$X$ along $\tau^{-1}(Q)$.
Therefore,  $\Sigma\subset \Di$.

Note that 
the set $\Sing(X)$ is given by the equations
\begin{equation}
\label{eq:sing:n1}
\begin{array}{l}
z=3y^2+\phi_4=0,
\\
y\frac{\partial \phi_4} {\partial x_i}+\frac{\partial \phi_6} {\partial x_i}=0,
\quad i=1,\, 2,\, 3,
\end{array}
\end{equation} 
while the set $\Sing(\Di)$ is given by the equations
\begin{equation}
\label{eq:sing:n2}
\begin{array}{l}
2\phi_4^2 \frac{\partial \phi_4} {\partial x_i}+ 9\phi_6\frac{\partial \phi_6} {\partial x_i}=0\quad i=1,\, 2,\, 3.
\end{array}
\end{equation} 

To prove \eqref{eq:singularities:n} suppose that
$Q:=(a_1,a_2,a_3)$ is a solution of \eqref{eq:sing:n2} with $\phi_6(Q)\neq 0$. Then putting $b:=\frac{2\phi_4(Q)^2}{9\phi_6(Q)}$
we see that $(a_1,a_2,a_3,b,0)$ is a solution of \eqref{eq:sing:n1}.
Conversely, 
let $(a_1,a_2,a_3,b,0)$ be a solution of \eqref{eq:sing:n1} with $\phi_6(Q)\neq 0$, where $Q:=(a_1,a_2,a_3)$.
Then $4\phi_4(Q)^3+27 \phi_6(Q)^2=0$, since $\Sigma\subset \Di$, and $b^2=-\frac 13 \phi_4(Q)\neq 0$ by \eqref{eq:sing:n1}. 
Hence using \eqref{eq:eq} we obtain
\begin{equation*}
b= -\frac{3\phi_6(Q) }{2\phi_4(Q)}=
\frac{3\phi_6(Q) }{2\phi_4(Q)} \cdot \frac {4\phi_4(Q)^3}{27 \phi_6(Q)^2}=
\frac {2\phi_4(Q)^2}{9 \phi_6(Q)}.
\end{equation*}
One can see that $Q=(a_1,a_2,a_3)$ is a solution of \eqref{eq:sing:n2}. This proves 
\eqref{eq:singularities:n}\footnote{See also \cite[Lemma~5.11]{ACPS:DoubleVC}}.

To prove \eqref{eq:singularities:3} 
we note that $\Lambda_6\cap \Sigma$ is contained in $\Lambda_4$ (because $\Lambda_6\cap \Sigma\subset \Di$).
Then one can see from \eqref{eq:sing:n1} that $\Lambda_6\cap \Sigma$ is given by $\phi_6=\phi_4=y=z=\partial \phi_6/\partial x_i=0$.
\end{proof}

\begin{corollary}
\label{cor:sing:comp:Di}
Let~$X$ be a double Veronese cone.
Then the curve $\Di$ has no multiple components outside $\Lambda_6$.
\end{corollary}

\begin{corollary}
\label{cor:w-blowup}
In the above notation, for any point $Q\in B$ either $\mult_Q(\Lambda_4)<4$ or $\mult_Q(\Lambda_6)<6$.
\end{corollary}

\begin{proof}
Assume the converse. 
Then near the singular point of $P\in X$ over $Q$ the local equation of~$X$ has the form 
\begin{equation*}
\bar z^2+\bar y^3+\bar y\phi_4^{\bullet}(\bar x_1,\bar x_2)+\phi_6^{\bullet}(\bar x_1,\bar x_2)=0
\end{equation*}
with $\ord_0(\phi_6^{\bullet})=6$ and $\ord_0(\phi_4^{\bullet})=4$ (or $\phi_4^{\bullet}=0$).
In this case, 
the weighted blowup with weights $(1,1,2,3)$ produces an exceptional divisor of discrepancy $0$.
Hence, the singularity $P\in X$ is not terminal, a contradiction.
\end{proof}

\begin{lemma}
\label{lemma:singularities:6}
Let $Q\in \Sigma$ and let $P$ be the singular point of~$X$ lying on $\tau^{-1}(Q)$.
Assume that $Q\in \Lambda_4\cap\Lambda_6$.
Then the following assertions hold:
\begin{enumerate}
\item 
$P\in X$ is a node \textup(ordinary double point\textup) if and only if $\Lambda_4\neq\varnothing$,
$Q$ is a smooth point of $\Lambda_4$, it is a double point of $\Lambda_6$, 
the curves $\Lambda_4$ and $\Lambda_6$ have no common components passing through $Q$, and
the local intersection number of $\Lambda_4$ and $\Lambda_6$ at $Q$ equals $(\Lambda_4\cdot \Lambda_6)_Q=2$;
\item 
$P\in X$ is of type \type{cA}{n} if and only if either $\Lambda_4\neq\varnothing$ and $\Lambda_4$ is smooth at $Q$ 
or $Q$ is a double point of~$\Lambda_6$.
\end{enumerate}
\end{lemma}

\begin{proof}
We may assume that 
$Q=(0,0,1)$ and in the affine chart $U=\{x_3\neq 0\}\simeq \Aff^4_{\bar x_1,\bar x_2,\bar y,\bar z}\subset \PP(1^3,2,3)$ the equation of~$X$ has the form
\begin{equation}
\label{eq:eq-chart}
\bar z^2+\bar y^3+\bar y\phi_4^{\bullet}(\bar x_1,\bar x_2)+\phi_6^{\bullet}(\bar x_1,\bar x_2)=0,
\end{equation} 
where $P$ is the origin, $\phi_4^{\bullet}=\phi_4(\bar x_1,\bar x_2,1)$, and $\phi_6^{\bullet}=\phi_6(\bar x_1,\bar x_2,1)$.
Since $P\in X$ is singular, \eqref{eq:eq-chart} has no linear terms.
By a linear coordinate change we may assume that
\begin{equation*}
\phi_4^{\bullet}=2ax_1+(\text{higher degree terms}),
\quad
\phi_6^{\bullet}=b_{1,1}x_1^2+2b_{1,2}x_1x_2+b_{2,2}x_2^2+(\text{higher degree terms}). 
\end{equation*}
The point $P\in X$ is not a node if and only if the quadratic part of \eqref{eq:eq-chart} is degenerate, i.e.
\begin{equation*}
\begin{vmatrix}
0 &a&0
\\
a& b_{1,1} &b_{1,2}
\\
0& b_{1,2} & b_{2,2}
\end{vmatrix} 
= -a^2b_{2,2}= 0.
\end{equation*}
The point $P\in X$ is not of type \type{cA}{} if and only if $a=b_{1,1}=b_{1,2}=b_{2,2}=0$.
This proves the lemma.
\end{proof}

We say that a double Veronese cone is \emph{nodal} if 
its singularities are nodes.

\begin{corollary}
\label{cor:nodal}
Let~$X$ be a double Veronese cone.
Assume that~$X$ is singular and nodal.
Then the following holds:
\begin{enumerate}

\item 
\label{cor:nodal:0} 
$\Lambda_4\neq\PP^2$;

\item 
\label{cor:nodal:CC} 
$\Lambda_4$ and $\Lambda_6$ have no common components;

\item 
\label{cor:nodal:1} 
if $Q\in \Lambda_4$ is a singular point, then it does not lie on $\Lambda_6$;

\item 
\label{cor:nodal:Lambda4} 
$\Lambda_4$ has no multiple components,
\item 
\label{cor:nodal:Lambda6} 
$\Lambda_6$ has no components of multiplicity $>2$.
\end{enumerate}
\end{corollary}

\begin{proposition}
\label{prop:finite}
Let~$X$ be a double Veronese cone.
Then there exists at most  finitely many  members $S\in |A_X|$ whose singularities are worse than Du Val.
\end{proposition}

\begin{proof}
Let $S\in |A_X|$ be a surface whose singularities are worse than Du Val
and let $l:=\tau(S)$.
To prove the assertion we may assume that~$l$ is not a component of $\Di$. 
If~$S$ is not normal, then its non-normal locus is a curve not contained in a fiber of $\tau$.
Hence
a general fiber of $\tau_S: S\to l$ is singular, 
which implies that~$l$ is a component of~$\Di$, a contradiction.

Thus we may assume that~$S$ is normal.
Then 
$S$ is a generalized cone over an elliptic curve~$C$ by Proposition~\ref{prop:hyp-sect}.
Let $P\in S$ be the vertex of the cone and let $Q=\tau(P)$.
Then the induced fibration $\tau_S: S\to l$ has exactly one singular fiber $\tau^{-1}(Q)$.
Therefore, the line~$l$ meets the discriminant $\Di$ at a single point $Q$.
If $Q$ is a smooth point of $\Di_{\red}$, then it is a flex point
and~$l$ is tangent to $\Di_{\red}$ at it.
There are only  finite number of such choices for the line~$l$.
Thus we may assume that $Q$ is a singular point of $\Di_{\red}$.
Since $l\cap \Di_{\red}=\{Q\}$, again there are only  
finite number of lines~$l$ provided
the multiplicity of the point $Q\in \Di_{\red}$ is less than $\deg \Di_{\red}$.
Thus we are only left with the case where $Q\in \Di_{\red}$ is a point of multiplicity $\deg \Di_{\red}$,
that is, $\Di_{\red}$ is a union of lines passing through $Q$.
By Corollary~\ref{cor:w-blowup} we have $\Lambda_4\neq\PP^2$, hence $\Supp(\Di)\neq \Supp(\Lambda_6)$.
Since~$S$ is a generalized cone, 
a general fiber of $\tau_S: S \dashrightarrow l$
is isomorphic to our elliptic curve~$C$. 
Therefore, the absolute invariant $\jj:B \dashrightarrow \PP^1$ is constant along the line $l\subset B=\PP^2$.
By \eqref{eq:j} this implies that for some $\lambda$ the polynomial $\phi_4^3+\lambda \phi_6^2$ vanishes along~$l$.
In other words,~$l$ is a component of a member of the pencil $\langle 3\Lambda_4,\, 2\Lambda_6\rangle$.
By Corollary~\ref{cor:sing:comp} the pencil $\langle 3\Lambda_4,\, 2\Lambda_6\rangle$
has no fixed components. But then its general member is irreducible and the line~$l$ can lie only in a finite number  
of reducible fibers.
\end{proof}

\section{Lines}
\label{sec:lines}

Throughout this section we assume that~$X$ is a del Pezzo threefold of arbitrary degree $d:=A_X^3$.
We are interested in the geometry of curves of degree~$1$ with respect to $A_X$, i.e. 
curves $L\subset X$ with $A_X\cdot L=1$.
Let $L\subset X$ be such a curve.
Clearly, $L$ is irreducible.
If $d\ge 3$, then $A_X$ is very ample (see~\cite{Shin1989}) and so $L$ is a usual projective line under the embedding given by $|A_X|$.
If $d=2$, then $|A_X|$ is base point free and defines a double cover $\tau:X\to \PP^3$ (again by~\cite{Shin1989})
so that $\tau_L: L\to \tau(L)$ is birational.
In this case $\tau(L)\subset \PP^3$ must be a line, hence $L$ is smooth and rational.
In the case $d=1$ the situation is a little more complicated:

\begin{lemma}
\label{lemma:L}
Let~$X$ be a double Veronese cone and let $L\subset X$ be a curve such that $A_X\cdot L=1$. Then one of the following holds:
\begin{enumerate}

\item
$L$ does not pass through the base point $O$ of $|A_X|$ and $L$ is a smooth rational curve;

\item
$L\ni O$ and $L$ is a fiber of the map $\tau: X \dashrightarrow \PP^2$ given by $|A_X|$; in particular, $\p(L)=1$.
\end{enumerate} 
\end{lemma}

\begin{proof}
If $L\not\ni O$, then $\tau$ is defined near $L$, the image $\tau(L)\subset \PP^2$ is a line, and the restriction $\tau_L:L\to \tau (L)$
is birational. 
As above, we obtain that $L$ is a smooth rational curve.
Assume that $L\ni O$. Since $A_X\cdot L=1$, there are two distinct members of $S_1,\, S_2\in |A_X|$ containing $L$.
Then $L=S_1\cap S_2$ and $\p(L)=1$ by the adjunction formula.
\end{proof}

A \textit{line} on a del Pezzo threefold~$X$ is a smooth rational curve $L$ such that $A_X\cdot L=1$.
Since $\dim |A_X|=d+1\ge 2$, any line $L\subset X$ is contained in a surface $S\in |A_X|$.
If $d=1$, such a surface is unique and coincides with $\tau^{-1}(\tau(L))$.

\begin{example}
Let~$X$ be a double Veronese cone and let $S\in |A_X|$ be any member. 
Then the family of lines lying on~$S$ has dimension $2$, $1$, and $\le 0$ 
in the cases 
\ref{prop:hyp-sect-c} \ref{prop:hyp-sect-b} \ref{prop:hyp-sect-a} of Proposition~\ref{prop:hyp-sect}, respectively.
\end{example}

A \textit{plane} on~$X$ is a surface $\Pi$ such that $A_X^2\cdot \Pi=1$ and in the case $d=1$ the base point $O$ does not lie on $\Pi$.

The main result of this section if the following
\begin{theorem}
\label{thm:lines}
Let~$X$ be a del Pezzo threefold. Then there exists a dense Zariski open subset $U\subset X$ such that
for any point $P\in U$ there exists a line passing through $P$ and lying in the smooth locus of~$X$.
\end{theorem}

\begin{remark}
If $L$ is a line corresponding to a general point of the covering family
(in particular, contained in the smooth locus of~$X$), then
the same arguments as in the proof of \cite[Ch.~III, Prop.~1.3]{Isk:anti-e} or \cite[Lemma~2.2.6]{KPS:Hilb} show that the normal bundle of $L$ is trivial, i.e.
$\NNN_{L/X}\simeq \OOO_L\oplus\OOO_L$.
\end{remark}

Note that in the above theorem it is not asserted that there is a line passing through \emph{any} smooth point of~$X$:

\begin{example}
\label{ex:lines:210}
Suppose that $X\subset \PP(1^3,2,3)$ is given by \eqref{eq:E8}. 
Then any member $S\in |A_X|$ whose image $\tau(S)$ passes through $Q=(0,0,1)$ is a normal del Pezzo surface of degree $1$ having a unique singularity 
of type \type{E}8. Then $\Cl(S)\simeq \ZZ\cdot K_X$, hence~$S$ does not contain lines.
Therefore, there are no lines on~$X$ meeting $\tau^{-1}(Q)$.
\end{example}

\begin{corollary}
\label{cor:lines}
Let~$X$ be a double Veronese cone.
For any point $Q\in B$ there is at most a one‑dimensional family of lines $L\subset X$ such that $\tau(L)\ni Q$.
\end{corollary}

\begin{proof}
Assume that contrary: the family of lines meeting $\tau^{-1}(Q)$ has dimension at least $2$.
Then the family of lines contained a general member $S\in |A_X|$ has dimension at least $1$.
This is possible only if the singularities of $S$ are worse than Du Val.
On the other hand, by Proposition~\ref{prop:finite} the number of such 
members $S\in |A_X|$ is finite.
\end{proof}

The proof of Theorem~\ref{thm:lines} is a consequence of Claims~\ref{claim:D-rkCl}, \ref{claim:rkCl>1}, and Corollary~\ref{corollary:bVc} below.

\begin{claim}
There is a line lying in the smooth locus of~$X$.
\end{claim}

\begin{proof}
One can take a $(-1)$-curve lying in a smooth member of $|A_X|$. 
\end{proof}

Recall that for any del Pezzo threefold~$X$ the linear system $|3A_X|$ is very ample and 
defines an embedding $X\hookrightarrow \PP^{10d+3}$. 
Let $\mathfrak{F}(X)$ be the Hilbert scheme of subschemes in~$X$ with Hilbert polynomial $3t+1$.
Clearly, $\mathfrak{F}(X)$ contains classes of lines in~$X$.
Note however that in the case $d=1$ the structure of $\mathfrak{F}(X)$ is quite complicated: apart from lines
it contains also the classes of schemes that are unions of a fiber of $\tau$ and a point (cf. \cite{Piene-Schlessinger}).

\begin{claim}
\label{claim:dim-Hilb}
If $L$ is a line lying in the smooth locus of~$X$, then the dimension of $\mathfrak{F}(X)$ at the corresponding point $[L]$ is at least $2$.
If furthermore $L$ is contained in a smooth member of $|A_X|$, 
then $\mathfrak{F}(X)$ is smooth at $[L]$ and $\dim_{[L]} \mathfrak{F}(X)= 2$.
\end{claim}

\begin{proof}
For the normal bundle of $L$ in~$X$ we have
\begin{equation*}
\det \NNN_{L/X} = -K_X\cdot L +\deg K_L= 0.
\end{equation*}
Therefore, by the deformation theory and Riemann-Roch 
\begin{equation*}
\dim_{[L]} \mathfrak{F}(X)\ge\dim H^0(\NNN_{L/X})- \dim H^1(\NNN_{L/X})= \det \NNN_{L/X} +2(1-\g(L))=2.
\end{equation*}
For the second assertion, let $S\in |A_X|$ be a smooth member containing $L$. Then~$S$ is a smooth del Pezzo surface that does not pass through $\Sing(X)$.
The same arguments as in \cite[\S 2]{KPS:Hilb}
show that
$\NNN_{L/X}\simeq \OOO_L\oplus \OOO_L$
or $\OOO_L(1)\oplus \OOO_L(-1)$. 
Therefore, $H^1(\NNN_{L/X})=0$ and $\dim H^0(\NNN_{L/X})=2$.
This implies that $\mathfrak{F}(X)$
is smooth and two-dimensional at $[L]$.
\end{proof}

\begin{notation}
\label{not:lines-1}
Let $L$ be a line contained in a smooth member~$S\in |A_X|$ and let $\mathfrak{F}\subset \mathfrak{F}(X)$ be a component that contains 
the class of $L$. 
Let $\mathfrak{U}$ be the universal family.
There 
are the following natural morphisms:
\begin{equation*}
\xymatrix{
&\mathfrak{U}\ar[dl]_{p}\ar[dr]^{q}&
\\
\mathfrak{F}& &X
} 
\end{equation*}
By Claim~\ref{claim:dim-Hilb} we have $\dim \mathfrak{F}=2$ and $\dim \mathfrak{U}=3$.
\end{notation}

\begin{claim}
\label{claim:pt}
For any point $P\in X$ the family of lines in~$X$ passing through $P$ has dimension at most $1$.
\end{claim}

In other words, the fibers of $q$ have dimension $\le 1$.

\begin{proof}
Assume the contrary: the family of lines
in~$X$ passing through $P$
has dimension at least~$2$.
First, consider the case $d\ge 3$. Then the linear system $|A_X|$ is very ample and defines an embedding 
$X\subset \PP^{d+1}$ \cite{Shin1989}. In this situation~$X$ must be a cone with vertex $P$. 
But then the point $P$ is not terminal, a contradiction.

Consider the case $d=2$. Then the linear system $|A_X|$ is base point free and defines 
a double cover $\pi: X\to \PP^3$ whose branch divisor $R\subset \PP^3$ is a quartic
with at worst isolated singularities \cite{Shin1989}.
For any line $L\subset X$ passing through $P$ its image $\pi(L)$ is a line on $\PP^3$ 
passing through $\pi(P)$. 
Therefore, by our assumption, the inverse image $\pi^{-1}(l)$ of a general line
$l\subset \PP^3$ passing through $P$ splits. This is possible only if 
no intersections $l\cap R$ are transversal, which, in turn, implies that 
$R$ is a cone with vertex $\pi(P)$ over a plane quartic curve.
But in this case the point $P$ is not terminal, a contradiction.

Finally, consider the case $d=1$. Then for a general line $l\subset B=\PP^2$
passing through $\tau(P)$ the surface $\tau^{-1}(l)$ contains at least a 
one-dimensional family of lines.
This contradicts Proposition~\ref{prop:finite}.
\end{proof}

\begin{claim}
\label{claim:D-rkCl}
The assertion of Theorem~\xref{thm:lines} holds true if $d\ge 2$ and $\rr(X)=1$. 
\end{claim}

\begin{proof}
We use the notation of \ref{not:lines-1}. By Claim~\ref{claim:dim-Hilb} we have $\dim\mathfrak{U}= 3$. 
It is sufficient to show that $q(\mathfrak{U})=X$. Assume the contrary, i.e. $\mathfrak{R}:=q(\mathfrak{U})$ is a surface in~$X$.
By our assumption $\Cl(X)\simeq \ZZ\cdot A_X$, hence $\mathfrak{R}\sim kA_X$ for some $k\ge 1$. 
Since $|2A_X|$ is very ample (see \cite{Shin1989}), the surface~$\mathfrak{R}$ contains a two dimensional family of conics under the corresponding embedding $X\subset \PP^{4d + 2}$.
According to
\cite[Lemma~A.1.2]{KPS:Hilb} the image of $\mathfrak{R}$ in $\PP^{4d + 2}$ is either the Veronese surface or its linear (regular or rational) projection.
In particular, $\deg \mathfrak{R}\le 4$. On the other hand, $\deg \mathfrak{R}=(2A_X)^2\cdot \mathfrak{R}=4a d\ge 8$, a contradiction.
\end{proof}

\begin{claim}
\label{claim:rkCl>1}
The assertion of Theorem~\xref{thm:lines} holds true if $\rr(X)>1$.
\end{claim}

The proof uses the following construction.

\begin{construction}[\cite{P:GFano1}, \cite{KP:dP}]
\label{construction}
Let~$X$ be a del Pezzo threefold with $\rr(X)>1$ and let $\xi: \hat{X}\to X$ be a $\QQ$-factorialization.
Then $\hat{X}$ is an almost del Pezzo threefold with $\uprho(\hat{X})=\rr(\hat{X})>1$ and there is an extremal Mori contraction $f: \hat{X}\to \hat{Z}$.
The structure of such contractions is described explicitly in \cite[Proposition~2.17]{KP:dP}.
There are the following possibilities:
\begin{enumerate}
\renewcommand\labelenumi{\rm (\Roman{enumi})}
\renewcommand\theenumi{\rm (\Roman{enumi})}

\item
\label{construction:1}
$\hat{Z}$ is a smooth del Pezzo surface and $f$ is a $\PP^1$-bundle;

\item
\label{construction:2}
$\hat{Z}\simeq \PP^1$ and $f$ is a quadric bundle;

\item
\label{construction:3}
$\hat{Z}$ is an almost del Pezzo threefold and $f$ is the blowup of a smooth point $P\in \hat{Z}$.
\end{enumerate}
In the latter case the image $\xi(E)$ of the $f$-exceptional divisor $E\subset \hat{X}$ is a plane. Thus there exists the following diagram
\begin{equation}
\label{eq:construction:di}
\vcenter{
\xymatrix{
\hat X\ar[r]^{f}\ar[d]^{\xi}& \hat{Z}\ar[d]^{\xi'}
\\
X  & Z
}}
\end{equation}
Here  $\xi': \hat{Z}\to Z$ is the map to the anticanonical model $Z=\hat Z_{\mathrm{can}}$
which is a del Pezzo threefold of degree $d+1$.

Conversely, if a del Pezzo threefold $X$ contains a plane $\Pi$, then there exists a 
$\QQ$-factorialization $\xi: \hat{X}\to X$ such that the proper transform of $\Pi$ is the exceptional divisor of a birational extremal Mori contraction
\cite[Lemma~3.8]{P:GFano1}, \cite[Proposition~3.12]{KP:dP}.
\end{construction}

\begin{proof}[Proof of Claim~\xref{claim:rkCl>1}]
We use the descending induction on $d$ and Claim~\ref{claim:D-rkCl}.
If $d=7$, then~$X$ the blowup of $\PP^3$ at a point \cite[A.1]{KP:dP}
(in particular,~$X$ is smooth).
In this case~$X$ has a structure of a $\PP^1$-bundle,
hence it is covered by lines. Thus we may assume that $d\le 6$.

Let $\xi: \hat{X}\to X$ be a $\QQ$-factorialization
(it is possible in the case $d=6$ that $\xi$ is an isomorphism). 
Apply Construction~\ref{construction}.

First, assume that $\dim \hat{Z}=2$. Then $f$ is a $\PP^1$-bundle. 
For a general geometric fiber $\Gamma$ of $f$ we have 
\begin{equation}
\label{eq:xiA}
\xi^* A_X\cdot \Gamma=-\frac12 K_{\hat{X}}\cdot \Gamma=1.
\end{equation} 
Hence~$X$ is covered by lines of the form $\xi(\Gamma)$ and a general such line does not pass through singular points of~$X$
by Claim~\ref{claim:pt}.
Assume that $\dim \hat{Z}=1$. Then $f$ is a quadric bundle.
A general ruling $\Gamma$ of a fiber $F\simeq \PP^1\times \PP^1$ of $f$
satisfies \eqref{eq:xiA} and,
as above, we see that~$X$ is covered by lines of the form $\xi(\Gamma)$ and a general such line 
does not pass through singular points of~$X$.

Therefore, we may assume that $f$ is birational.
Then $\hat{Z}$ is an almost del Pezzo threefold of degree $d'=d+1$ 
and we have the diagram~\eqref{eq:construction:di}.
By the induction hypothesis and Claim~\ref{claim:D-rkCl}
the del Pezzo threefold $Z$ is covered by lines.
Then, as above, the proper transform $\Gamma_X\subset X$ of a line $\Gamma\subset Z$
is a line and such lines cover~$X$.
\end{proof}

It remains to consider the case $d=\rr(X)=1$, i.e. the case of $\QQ$-factorial double Veronese cones.

\begin{claim}
\label{claim:bVc}
Let~$X$ be a double Veronese cone.
Suppose that $S\subset X$ is a surface containing a two‑dimensional family of lines.
Then~$S$ is either a plane or a non‑normal member of $|A_X|$.
\end{claim}

\begin{proof}
Recall that the linear system $|2A_X|$ defines a double cover $\pi: X\to Y\subset \PP^6$, where $Y$ is the cone over the Veronese surface
$\upsilon_2(\PP^2)\subset \PP^5$ with vertex $\pi(O)$.
Let $L\subset S$ be a line.
If $\deg \pi(L)=1$, then $\pi(L)$ is a line passing through $\pi(O)$.
But then $L$ must pass through $O$, which is impossible by Lemma~\ref{lemma:L}. 
Therefore, $\deg\pi(L)=A_X\cdot L=2$, i.e. $\pi(L)$ is a smooth conic and $\pi_L: L\to \pi(L)$ is an isomorphism. 
Thus
$\pi(S)\subset Y\subset \PP^6$ is an irreducible surface containing a two-dimensional family of 
conics. According to 
\cite[Lemma~A.1.2]{KPS:Hilb} the image $\pi(S)$ is either the Veronese surface or its linear (regular or rational) projection.
In particular, $\deg \pi(S)\le 4$. On the other hand, 
\begin{equation*}
8\ge 2\deg \pi(S)\ge \deg (\pi_S)\cdot \deg \pi(S)=(2A_X)^2\cdot S= 4 A_X^2\cdot S\ge 4.
\end{equation*}
In particular, $A_X^2\cdot S\le 2$.
Since $Y$ does not contain any planes, $\deg \pi(S) \ge 2$. We obtain 
either $\deg \pi(S)=4$ or $\deg \pi(S)=2$.
Assume that $\deg \pi(S)=2$. Then $\pi(S)$ must be a quadratic cone whose vertex coincides with
$\pi(O)$. In this case $S\sim A_X$ and~$S$ is non-normal by Proposition~\ref{prop:hyp-sect}.
Thus we may assume that $\deg \pi(S)=4$. Then $\pi(S)$ is the Veronese surface and it does not 
pass through $\pi(O)$. 
Thus $(2A_X)^2\cdot S=\deg \pi(S)=4$, hence $A_X^2\cdot S=1$ and the morphism $\pi_S: S\to \pi(S)$ is birational.
Hence $S$ is a plane and $\pi_S$ is an isomorphism.
\end{proof}

\begin{corollary}
\label{corollary:bVc}
The assertion of Theorem~\xref{thm:lines} holds true if $d=\rr(X)=1$.
\end{corollary}

\begin{proof}
In the notation \ref{not:lines-1} we can take $L$ so that $\tau(L)\not \subset \Di$.
Then $L$ is not contained in a non-normal member of $|A_X|$
and~$X$ does not contain any planes because it is $\QQ$-factorial.
Therefore, we have $\dim q(\mathfrak{U})=3$ by Claim~\ref{claim:bVc}.
\end{proof}

The last corollary concludes the proof of Theorem~\ref{thm:lines}.
Using the same arguments we can prove the following criterion of $\QQ$-factoriality.

\begin{proposition}
\label{prop:nQfact}
Let~$X$ be a del Pezzo threefold. 
Suppose that $X$ contains a non-$\QQ$-factorial point $P$.
Then there is a one-dimensional family of lines on $X$ passing through $P$. 
\end{proposition}

\begin{proof}
If $A_X^3\ge 6$, then there is only one possibility: 
$X$ is a hypersurface in $\PP^2_{x_1,x_2,x_3}\times \PP^2_{y_1,y_2,y_3}$ given by $x_1y_1+x_2y_2=0$ (see \cite[\S~7.7]{P:GFano1} or \cite[Lemma~5.2, \S~A.1.4]{KP:dP}). 
It is easy to see that the 
assertion holds true in this case. Thus we may assume that $A_X^3\le 5$.

Apply Construction~\ref{construction}. Then  $C:=\xi^{-1}(P)$
is a (possibly reducible) connected curve and $\hat O:=\xi^{-1}(O)$
is a point that does not lie on~$C$.
Note that the morphism $f$ does not contract any component of~$C$.
If  $f$ is birational and its exceptional divisor $E$ meets~$C$, 
then $\xi(E)$ is a plane passing through $P$ and we are done.
If $f$ is not birational, then, as in the proof of Claim~\xref{claim:rkCl>1},
there is  a one-dimensional family of smooth rational curves  $\Gamma_\lambda$ 
meeting~$C$ such that $\xi^* A_X\cdot \Gamma_\lambda=1$.
Moreover, a general such a curve does not pass through $\hat O$.
This proves the assertion in this case.
Finally assume that $f$ is birational and it is an isomorphism near~$C$.
Then the morphism $\xi'$ in \eqref{eq:construction:di} contracts $f(C)$ 
to a non-$\QQ$-factorial point. As in the proof of Claim~\xref{claim:rkCl>1}
we can proceed by descending induction on  $A_X^3$.
\end{proof}

\section{The class group and $\QQ$-factoriality}
\label{sec:class}

\subsection{The class group}
For the facts presented in this subsection we refer to \cite{P:GFano1} and \cite{KP:dP}.
Let~$X$ be a double Veronese cone.
Then $\Cl(X)$ is a free abelian group \cite[Corollary~3.4]{KP:dP} equipped with a symmetric bilinear form
\begin{equation*}
\langle D_1,\, D_2\rangle := D_1\cdot D_2\cdot A_X. 
\end{equation*}
This form is non‑degenerate and has signature $(1,\rr(X)-1)$, where $\rr(X):= \rk\Cl(X)$. Denote
\begin{equation*}
\boldsymbol{\Delta}(X):=\{ \bx\in \Cl(X) \mid \langle A_X,\bx\rangle=0,\ \langle \bx,\, \bx\rangle=-2\}. 
\end{equation*}
If $\boldsymbol{\Delta}(X)$ is non‑empty, it is a root system in the negative definite lattice $A_X^\perp \subset\Cl(X)$.
Note however that $\boldsymbol{\Delta}(X)$ not necessarily has maximal rank $\rr(X)-1$ in $A_X^\perp$.
There is a natural bijection between the planes on~$X$ and the elements of the root system $\boldsymbol{\Delta}(X)$:
\begin{equation*}
\Pi \longmapsto \Pi- A_X. 
\end{equation*}

Let $S\in |A_X|$ be a general member and let $i:S \hookrightarrow X$ be the embedding. 
The natural restriction homomorphism $i^*:\Cl(X)\to \Cl(S)$ is injective \cite[Proposition~3.3]{KP:dP}.
The orthogonal complement $\boldsymbol{\Xi}(X):= \Cl(X)^{\perp}\subset \Cl(S)$ is a negative definite lattice of rank $9-\rr(X)$. The lattice $\boldsymbol{\Xi}(X)$ does not depend on the choice 
of~$S$ and has a root system 
of full rank. The type of $\boldsymbol{\Xi}(X)$ is called the \textit{Dynkin type} of~$X$.
The properties of the lattices $\Cl(X)$ are summarized in Table~\ref{tab}. Here 
\begin{equation*}
\Theta_2:=\left\{ \bx\in \Cl(X) \mid \langle A_X,\bx\rangle=2,\ \langle \bx,\, \bx\rangle=0\right\}
\end{equation*}
is the set of $\PP^1$-classes (they describe the so-called primitive models of~$X$ that are quadric bundles over $\PP^1$), and 
\begin{equation*}
\Theta_3^\circ:=\left\{ \bx\in \Cl(X) \mid \langle A_X,\bx\rangle=3,\ \langle \bx,\, \bx\rangle=1,\ \bx-A_X\notin 2\Cl(X) \right\}
\end{equation*}
is the set of $\PP^2$-classes (they describe so-called primitive models of~$X$ that are $\PP^1$-bundles over $\PP^2$); see \cite{P:GFano1} and \cite{KP:dP} for details.

\newcolumntype{Y}{>{\centering\arraybackslash}X}

\begin{table}[h]
\setlongtables \renewcommand{\arraystretch}{1.2}
\begin{tabularx}{0.9\textwidth}{|Y|c|c|c|c|c|c|c|c|}
\hline
$\boldsymbol{\Xi}(X)$, Dynkin type  & $\rr(X)$ & $\boldsymbol{\Delta}(X)$ & $\rk \boldsymbol{\Delta}(X)$ & $|\boldsymbol{\Delta}(X)|$ & $|\Theta_2|$ & $|\Theta_3^\circ|$
\tabularnewline
\hline
\type{E}{8}& $1$ & $\varnothing$ & $0$ & $0$ & $0$ & $0$ 
\tabularnewline
\type{E}{7} & $2$ & \type{A}{1} & $1$ & $2$ & $0$ & $0$ 
\tabularnewline
\type{D}{7} & $2$ & $\varnothing$ & $0$ & $0$ & $2 $ & $0$
\tabularnewline
\type{A}{7} & $2$ & $\varnothing$ & $0$ & $0$ & $0$ & $2$ 
\tabularnewline
\type{E}{6} & $3$ & \type{A}{2} & $2$ & $6$ & $0$ & $0$ 
\tabularnewline
\type{D}{6} & $3$ & \type{A}{1}$\times$\type{A}{1} & $2$ & $4$ & $4$ & $0$ 
\tabularnewline
\type{A}{6} & $3$ & \type{A}{1} & $1$ & $2$ & $4$ & $4$ 
\tabularnewline
\type{D}{5} & $4$ & \type{A}{3} & $3$ & $12$ & $6 $ & $0$ 
\tabularnewline
\type{A}{5} & $4$ & \type{A}{1}$\times$\type{A}{2} & $3$ & $8$ & $12 $ & $12$ 
\tabularnewline
\type{D}{4} & $5$ & \type{D}{4} & $4$ & $24$ & $24$ & $0$ 
\tabularnewline
\type{A}{4} & $5$ & \type{A}{4} & $4$ & $20$ & $30$ & $40$
\tabularnewline
\type{A}{3} & $6$ & \type{D}{5} & $5$ & $40$ & $90 $ & $160 $ 
\tabularnewline
\type{A}{2} & $7$ & \type{E}{6} & $6$ & $72$ & $270 $ & $864$ 
\tabularnewline
\type{A}{1} & $8$ & \type{E}{7} & $7$ & $126$ & $756 $ & $4032$ 
\tabularnewline
\hline
\end{tabularx}
\caption{Lattices $\Cl(X)$ and $\boldsymbol{\Xi}(X)$}
\label{tab}
\end{table}

In the rest of this section we discuss $\QQ$-factoriality of double Veronese cones
with relation to the number of singular points.

\begin{example}
\label{ex:Q-fact}
The double Veronese cone given by \eqref{eq:E8} does not contain lines passing through 
the (unique) singular point (see Example~\ref{ex:lines:210}). 
Then it is $\QQ$-factorial by Proposition~\ref{prop:nQfact}.
\end{example}

If $X$ is a variety whose singularities are only nodes, then by 
\textit{economic resolution} of singularities we mean just blowup of all singular points.

\begin{theorem}
\label{thm:Q-fact}
Let~$X$ be a nodal double Veronese cone.
\begin{enumerate}

\item 
\label{prop:singul:1} 
We have $|\Sing(X)|\le 28$ and the equality holds only if $\rr(X)=8$ and 
$\hh(\tilde{X})=0$, where $\tilde{X}$ is the economic resolution of singularities of~$X$.

\item
\label{thm:Q-fact:f}
If~$X$ is $\QQ$-factorial, then $|\Sing(X)| \le 21$, and the equality $|\Sing(X)| = 21$ is attained for some $\QQ$-factorial nodal double Veronese cone~$X$.

\item
\label{thm:Q-fact:n}
If~$X$ is not $\QQ$-factorial, then $|\Sing(X)| \ge 12$, and the equality $|\Sing(X)| = 12$ is attained for some non-$\QQ$-factorial nodal double Veronese cone~$X$.
\end{enumerate}
\end{theorem}

Similar bounds for other types of Fano threefolds of Picard number one are obtained in \cite{P:factorial-Fano:e} and \cite{Cheltsov2005a}.

\begin{proposition}[{\cite [Corollary~10.9]{P:GFano1}}]
\label{prop:Sing:r}
Let $V$ be a three‑dimensional variety whose singularities are 
nodes. 
Assume that $H^i(V,\OOO_V)=0$ for $i=1,\,2,\,3$. 
Furthermore, assume that there exists an embedding of $V$ into a smooth fourfold $W$ such that a general member $V^{\mathrm{sm}}$ of the linear system $|V|$ is smooth.
Then
\begin{equation}
\label{eq:n-Sing}
|\Sing(V)|=\rr(V) - \uprho(V)+\hh(V^{\mathrm{sm}}) - \hh(\tilde{V} ),
\end{equation} 
where $\tilde{V}$ is the economic resolution of singularities of $V$.
\end{proposition}

\begin{proof}[Proof of Theorem~\xref{thm:Q-fact}]
A general member $X^{\mathrm{sm}}$ of the linear system $|X|$ on $\PP(1^3,2,3)$ is a smooth double Veronese cone and 
$\hh(X^{\mathrm{sm}})=21$ \cite[\S~12.2]{IP99}.
Hence, according to~\eqref{eq:n-Sing} we have 
\begin{equation*}
| \Sing(X)|= 20+\rr(X)- \hh(\tilde{X})= 20+\rr(X)- \hh(\tilde{X}),
\end{equation*}
where $\tilde{X}$ is the economic resolution of singularities of~$X$.
Since $\rr(X)\le 8$, by Table~\ref{tab}, this proves \ref{prop:singul:1} and the first assertion in \ref{thm:Q-fact:f}.
The second assertion in~\ref{thm:Q-fact:f} follows from Example~\ref{ex:PSL27} below.

Now we prove \ref{thm:Q-fact:n}. So, we assume that~$X$ is not $\QQ$-factorial.
Apply Construction~\ref{construction}.
Thus we let $\xi:\hat{X}\to X$ be a $\QQ$-factorialization and $\varphi:\hat X\to Z$ be an extremal Mori contraction.
There are the following possibilities.

\subsection*{Case $\dim Z=2$.}
Then $Z$ is a smooth del Pezzo surface and $\varphi$ is a $\PP^1$-bundle.
In particular, $\hat X$ is smooth.
In this case 
$\hh(\hat{X})=0$ and $|\Sing(X)|= 21+\uprho(Z)\ge 22$ by Proposition~\ref{prop:Sing:r}.

\subsection*{Case $\dim Z=1$.}
Then $Z\simeq \PP^1$ and $\varphi$ is a quadric bundle.
By \cite[Theorem~1.11]{KP:dP} the variety $\hat{X}$ can be realized as a complete intersection in $\hat{X}\subset \PP^{1}\times \PP^{6}$ of three divisors $D_1$, $D_2$, $D_3$ of bidegree $(1,1)$ 
and one divisor $D$ of bidegree $(1,2)$
\cite[Proposition~4.14]{KP:dP}, \cite[Theorem~2.3]{Takeuchi:DP}.
Let $h_1$ (resp. $h_2$) be the pull-back on $\PP^{1}\times \PP^{6}$
of the hyperplane class of $\PP^1$ (resp. $\PP^6$). 

First, assume that 
$\hat{X}$ is smooth. From the exact sequence 
\begin{equation*}
0 \longrightarrow \TTT_{\hat{X}} \longrightarrow \TTT_{\PP^1\times \PP^6}|_X \longrightarrow \NNN_{\hat{X}/\PP^1\times \PP^6} \longrightarrow 0
\end{equation*}
we conclude that the Chern polynomial of $\hat{X}$ has the form
\begin{equation*}
\cc_t(\hat{X})=\left.\frac{\cc_t(\PP^1\times \PP^6)}{ \cc_t(\NNN_{\hat{X}/\PP^1\times \PP^6})}\right|_{\hat{X}}
=
\left.\frac{(1+h_1t)^2\cdot (1+h_2t)^7}{ (1+(h_1+h_2)t)^{3}\cdot (1+(h_1+2h_2)t)}\right|_{\hat{X}}.
\end{equation*}
From this we obtain
\begin{equation*}
\chit(\hat{X})= \cc_3(\hat{X})= -2h_2^6\cdot h_{1}=-2\qquad\text{and}\qquad \hh (\hat{X})=4.
\end{equation*}
Then the assertion follows from Proposition~\ref{prop:Sing:r}.
Now assume that $\hat{X}$ is singular and let $\sigma: \tilde X\to \hat{X}$ 
be the blowup of all singular points.
Again by Proposition~\ref{prop:Sing:r} we have 
\begin{equation*}
\hh(\tilde X) = \hh(\hat{X}^{\mathrm{sm}}) - |\Sing(\hat{X})|= 4 - |\Sing(\hat{X})|,
\end{equation*}
where $\hat{X}^{\mathrm{sm}}$ is a general smooth member of the linear system $|\hat{X}|$
on $Z$. Therefore,
\begin{equation*}
|\Sing(X)|=22 - \hh(\tilde X ) =18+|\Sing(\hat{X})|.
\end{equation*}

\subsection*{Case $\dim Z=3$.}
Then by Construction~\ref{construction} we have the diagram \eqref{eq:construction:di}
where 
$\varphi$ is the blowup of a smooth point,
$\hat Z$ is an almost del Pezzo threefold of degree~$2$,
and $\xi'$ is the morphism to the anticanonical model $Z=\hat Z_{\mathrm{can}}$
which is a del Pezzo threefold of degree~$2$.
Clearly, $\hat{X}$, $Z$, and $\hat Z$ have only nodal singularities.
Moreover, the economic resolution $\tilde X$ of~$X$ is obtained from the economic resolution 
$\tilde Z$ of $Z$
by blowing up only points and smooth rational curves.
Hence, $\hh(\tilde X)=\hh(\tilde Z)$.
For the smoothing $Z'$ of $Z$ we have $\hh(Z')=10$ (see \cite[\S~12.2]{IP99}).
Therefore, Proposition~\ref{prop:Sing:r} we have 
\begin{equation*}
|\Sing(X)|-20=\rr(X)-\hh(X)= \rr(Z_{\mathrm{can}})-\hh(\tilde Z_{\mathrm{can}})+1=|\Sing(Z_{\mathrm{can}})|+1-9.
\end{equation*}
We obtain $|\Sing(X)|=12+|\Sing(Z_{\mathrm{can}})|\ge 12$.
The last assertion in \ref{thm:Q-fact:n} follows from the example below.
\end{proof}

\begin{example}
Let $X\subset \PP(1^3,2,3)$ be given by the equation
\begin{equation}
\label{eq:E8-1}
z^2+y^3+\phi_4(x_1,x_2,x_3)y-\phi_3(x_1,x_2,x_3)^2=0,
\end{equation} 
where $\phi_4$ and $\phi_3$ are general homogeneous polynomials of degree~$4$ and $3$, respectively.
Then~$X$ is a double Veronese cone containing two planes $\{y=z\pm \phi_3(x_1,x_2,x_3)=0\}$.
The singular locus of~$X$ consists of $12$ nodes at $\{z=y=\phi_4=\phi_3=0\}$.
\end{example}

Now we give an example of $\QQ$-factorial double Veronese cone having $21$ nodes.
\begin{example}
\label{ex:PSL27}
Let $G:=\PSL_2(\bF_7)$ be the Klein simple group of order $168$.
This group has two three-dimensional complex representations.
They define a unique conjugacy class of embeddings $G\subset \SL_3(\Bbbk)$, hence $G$ faithfully acts on
$\Bbbk[x_1,x_2,x_3]$ and on $\PP^2$.
The ring of invariants $\Bbbk[x_1,x_2,x_3]^G$
is generated by four homogeneous polynomials $\phi_4$, $\phi_6$, $\phi_{14}$, $\phi_{21}$,
where $\deg \phi_d=d$ (see {\cite{Klein1878:e}} or {\cite[\S\S~139-140]{Weber:Algebra2}}).
In some coordinate system the invariants $\phi_4$ and $\phi_6$ have the form
\begin{eqnarray}
\label{eq:phi4}
\phi_4&=&x_1^3x_3+x_2^3x_1+x_3^3x_2,
\\
\label{eq:phi6}
\phi_6&=&5 x_1^2x_2^2x_3^2-x_1^5x_2-x_2^5x_3-x_3^5x_1.
\end{eqnarray}
Now let~$X$ be given by the equation
\begin{equation}
\label{eq:eq-7}
z^2+y^3+3\phi_4(x_1,x_2,x_3) y+\phi_6(x_1,x_2,x_3)=0.
\end{equation}
The $G$-action on $\Bbbk[x_1,x_2,x_3]$ extends to~$X$, where the action on $y$ and $z$ are supposed to be trivial.
In this case the discriminant $\Di$ is given by $4\phi_4^3+\phi_6^2=0$.

By using Macaulay2 \cite{M2} we check that~$X$ has exactly 21 singular points and all these points are nodes
(see Appendix~\ref{sect:code}).
It remains to show that~$X$ is $\QQ$-factorial. Assume the contrary: $\rr(X)\ge 2$. 
By Table~\ref{tab} we have $\rr(X)\le 8$.
Moreover, if $\rr(X)=8$, then the number of singular points equals $28$ by \cite[Theorem~7.1]{P:GFano1}.
Therefore, $\rr(X)<8$. 

First, consider the case where $G$ faithfully acts on $\Cl(X)$.
Then the representation of $G$ on the orthogonal complement $K_X^\perp\subset \Cl(X)\otimes \QQ$
if faithful. Since the group $G$ has no faithful representations 
of degree~$\le 5$ defined over $\QQ$, we have $\dim K_X^\perp\ge 6$ and the only possibility is $\rr(X)=7$.
Then the root system $\boldsymbol{\Delta}(X)$ is of type \type{E}{6} 
(see Table~\ref{tab}). However, in this case the automorphism group of $\boldsymbol{\Delta}(X)$
has order $4\cdot 25920$ \cite[Plate~V]{Bourbaki:Lie:4-6}, hence $\Aut(\boldsymbol{\Delta}(X))$ does not contain subgroups isomorphic to $G=\PSL_2(\bF_7)$,
a contradiction.

Now, assume that the action of $G$ on $\Cl(X)$ is trivial.
Then in the notation of Construction~\ref{construction} the $G$-action  lifts to $\hat X$ and the contraction $f: \hat X\to\hat  Z$
is $G$-equivariant.
Note that $\xi$ is an isomorphism over $O$ and the point $\hat O:=\xi^{-1}(O)$ is fixed by $G$. 
If $f$ is a quadric bundle over $\PP^1$, then $G$ must act faithfully either on the base $\hat  Z=\PP^1$ or 
on the general geometric fiber, which is isomorphic to $\PP^1\times \PP^1$. Clearly, this is impossible.
If $f$ is a $\PP^1$-bundle, then $\hat Z$ is a smooth surface admitting an effective $G$-action 
and having a fixed point $f(\hat O)$. But then the induced two-dimensional representation of $G$ in $T_{f(\hat O), \hat Z}$
must be faithful, a contradiction.
Finally, suppose that $f$ be birational,
Then the $f$-exceptional divisor is $G$-invariant and its image $\Pi\subset X$ is a plane,
which is also $G$-invariant. But then the ramification divisor $R_X$ of the double cover 
$\pi: X\to Y$ cuts out on $\Pi$ a $G$-invariant cubic curve (because $R_X\sim 3A_X$ and $R_X$ is irreducible).
Since the Klein simple group $G$ cannot faithfully act on a curve of genus $\le 1$,
the induced $G$-action on $\Pi$ is trivial.
Since the restriction $\tau_{\Pi}: \Pi\to B$ is dominant, the $G$-action on $B$ is trivial as well, a contradiction.
\end{example}

Note that, similar to \eqref{eq:eq-7}, there exists a pencil of double Veronese cones 
\begin{equation*}
z^2+y^3+\lambda \phi_4(x_1,x_2,x_3) y+\phi_6(x_1,x_2,x_3)=0
\end{equation*}
admitting actions of the Klein simple group $\PSL_2(\bF_7)$.
This  will discussed  in the forthcoming paper \cite{P:V1:Klein}.

\section{The automorphism group}
\label{sec:aut}

The automorphism group of a Fano variety is a linear algebraic group (see e.g. \cite[Theorem~1.3(vii)]{P:Fano25e}).
In the case of double Veronese cones we can say more.
We say that a double Veronese cone~$X$ is \textit{equianharmonic} if the absolute invariant of the general fiber equals $0$, i.e. in the presentation  \eqref{eq:eq} we have $\phi_4=0$.
Note that the \textit{harmonic} case $\phi_6=0$ does not occur in our situation because in this case 
the singularities of $X$ are not isolated, hence they are not terminal. 

\begin{proposition}
Let~$X$ be a double Veronese cone. The natural homomorphism 
\begin{equation}
\label{eq:homo:PGL:0}
\Aut(X)\longrightarrow \GL(T_{O,X})
\end{equation} 
is injective. Let $\Aut(X)_K$ be the kernel of the induced homomorphism
\begin{equation}
\label{eq:homo:PGL}
\Aut(X)\longrightarrow \Aut(\PP(T_{O,X})).
\end{equation} 
Then 
\begin{equation}
\label{eq:homo:PGL:1}
\Aut(X)_K\simeq 
\begin{cases}
\ZZ/6\ZZ
&\text{if~$X$ is equianharmonic},\\
\ZZ/2\ZZ
&\text{otherwise}.
\end{cases}
\end{equation} 
\end{proposition}

\begin{proof}[Sketch of the proof]
The projectivization $\PP(T_{O,X})$ is naturally identified with the exceptional divisor
of the blowup $\sigma: \tilde X\to X$ and with the base of the elliptic fibration 
$\varphi: \tilde X\to B=\PP^2$. Hence $\Aut(X)_K$ acts on the general fiber $X_\eta$ of $\varphi$ 
faithfully, that is, $\Aut(X)_K\subset \Aut(X_\eta)$, where the group $\Aut(X_\eta)$
is described by the right hand side of \eqref{eq:homo:PGL:1}. 
On the other hand, $\Aut(X)_K$ always contains the Bertini involution (see \eqref{eq:eq:Bertini})
and in the case $\phi_4=0$ by  \eqref{eq:eq} it contains the element of order $3$ acting via 
\begin{equation*}
(x_1,x_2,x_3,y,z) \longmapsto (x_1,x_2,x_3,\zeta_3 y,z).
\end{equation*}
This proves the last assertion.
This implies also that the kernel of the homomorphism \eqref{eq:homo:PGL:0} is finite, hence it is trivial
(see e.g. \cite[\S~2.2]{Akhiezer1995}). 
\end{proof}

Thus we have an exact sequence 
\begin{equation}
\label{eq:exact-sequence}
1 \longrightarrow \Aut(X)_K\longrightarrow \Aut(X)\longrightarrow\Aut(X)_B \longrightarrow 1,
\end{equation} 
where $\Aut(X)_B$ is the image of $\Aut(X)$ in $\PGL_3(\Bbbk)$ under the homomorphism \eqref{eq:homo:PGL}.
The group $\Aut(X)_B$ coincides with the image of 
the induced action of $\Aut(X)$ on the base of the elliptic fibration $\varphi: \tilde X\to B=\PP^2$.
By construction all the maps in the diagram~\eqref{eq:preV1} 
are $\Aut(X)$-equivariant.
Note also that the curves $\Di$, $\Lambda_6$, and $\Lambda_4$ (if $\Lambda_4\neq\PP^2$)
are invariant with respect to the action of $\Aut(X)_B$ (see e.g. \eqref{eq:LambdaLambdaDi}).

Let $\Aut(X)^0$ be the identity component of $\Aut(X)$.
Then~\eqref{eq:exact-sequence} induces 
an exact sequence
\begin{equation}
\label{eq:exact-sec}
1\longrightarrow (\Aut(X)^0)_K \longrightarrow \Aut(X)^0 \longrightarrow \Aut(X)^0_B \longrightarrow 1, 
\end{equation}
where $(\Aut(X)^0)_K:=\Aut(X)_K\cap \Aut(X)^0$ is a finite group and 
$\Aut(X)^0_B$ is a connected linear algebraic group that acts on $B=\PP^2$ faithfully. 

In the case of a smooth double Veronese cone the group $\Aut(X)$ is finite (see e.g. \cite[Lemma~4.4.1]{KPS:Hilb}).
Below we discuss the finiteness of the automorphism group of singular double Veronese cones.

\begin{theorem}
\label{thm:aut}
Let~$X$ be a double Veronese cone whose automorphism group $\Aut(X)$ is infinite. Then the identity component $\Aut(X)^0\subset \Aut(X)$
is isomorphic to either $\Gm$ or $\Ga$.
The explicit equations of~$X$ and corresponding actions are described in 
Proposition~\xref{prop:torus} and Proposition~\xref{prop:Ga}, respectively.
If~$X$ is nodal, then $\Aut(X)$ is finite.
\end{theorem}

The rest of this section is devoted to the proof of Theorem~\ref{thm:aut}.
Throughout this section we assume that $X$ is a double Veronese cone.

\begin{lemma}
\label{lemma:group}
Assume that the group $\Aut(X)$ is infinite.
Then one of the following holds:
\begin{enumerate}
\item 
\label{lemma:group:Gm}
$\Aut(X)^0\simeq\Gm$,
\item 
\label{lemma:group:Ga}
$\Aut(X)^0\simeq\Ga$,
\item 
\label{lemma:group:GaGm}
$\Aut(X)^0$ is isomorphic to 
a semi-direct product $\Ga \rtimes \Gm$ \textup(for some homomorphism $\Gm\to \Aut(\Ga)$\textup).
\end{enumerate}
\end{lemma}

\begin{proof}Since the group $\Aut(X)_K$ is finite, 
it follows from \eqref{eq:exact-sequence} that 
$\dim \Aut(X)=\dim \Aut(X)_B$.
Let $C=\cup C_i$ be the union of components of $\Di$, $\Lambda_6$, and $\Lambda_4$ 
(if $\Lambda_4\neq \PP^2$). Then~$C$ is 
an $\Aut(X)_B$-invariant plane curve and every singular point of~$C$
is fixed by $\Aut(X)^0_B$. 
We may assume that $\dim \Aut(X)^0\ge 2$ (otherwise we are in  the case~\ref{lemma:group:Gm} or \ref{lemma:group:Ga}).

First, consider the case where all the components of~$C$ are lines.
If all these lines pass through one point $Q\in \PP^2$, then
$\mult_Q(\Lambda_6)=6$ and $\mult_Q(\Lambda_4)=4$
(if $\Lambda_4\neq \PP^2$). This contradicts Corollary~\ref{cor:w-blowup}.
Thus, $C$ has three components forming a triangle: 
the union of three lines $C_1,C_2,C_3$ in general position. 
Then $\Aut(X)^0_B$
must be an algebraic torus and $\Aut(X)^0$ must be  an algebraic torus as well.
By our assumption $\dim \Aut(X)^0_B>1$. Then $\Aut(X)^0_B\simeq (\Gm)^2$ and 
$C=C_1+C_2+C_3$.
We may assume that the equations of the lines $C_1,C_2,C_3$ are $x_1=0$, $x_2=0$, and $x_3=0$, respectively.
Then $\phi_6=x_1^{k_1}x_2^{k_2}x_3^{k_3}$ with $k_1+k_2+k_3=6$ and $\phi_4=x_1^{l_1}x_2^{l_2}x_3^{l_3}$ with $l_1+l_2+l_3=4$.
On the other hand, 
\begin{equation*}
4\phi_4^3+27\phi_6^2= 4x_1^{3l_1}x_2^{3l_2}x_3^{3l_3}+27x_1^{2k_1}x_2^{2k_2}x_3^{2k_3}
\end{equation*}
is a semi-invariant of $\Aut(X)^0_B$. Hence, $3l_i=2k_i$ for $i=1,2,3$ and so $k_i\ge 3$
whenever $l_i\ge 1$. This contradicts Lemma~\ref{lemma:sing0}.

Thus we may assume that~$C$ has a component, say $C_1$, that is not a line.
Then the connected linear algebraic group
$\Aut(X)^0_B$ acts on $C_1$ faithfully.
Therefore, $C_1$ is rational and so $\Aut(X)^0_B\subset \Aut(C_1)\subset \PGL_2(\Bbbk)$.

If $\Aut(X)^0_B$ has no fixed points on $C_1$, then $C_1$ is smooth and does not meet other components of~$C$.
In this case  $C_1$ must be a conic, $C=C_1$, $\Lambda_6=3C_1$, and
$\Lambda_4=2C_1$ (or $\Lambda_4=\PP^2$).
This contradicts Corollary~\ref{cor:sing:comp}.

Therefore, $\Aut(X)^0_B$ has a fixed point on $C_1$. Then 
$\Aut(X)^0_B$ is 
Borel subgroup of $\PGL_2(\Bbbk)$. 
Thus the unipotent radical of $\Aut(X)^0_B$ is isomorphic to $\Ga$
and $\Aut(X)^0_B$ contains a one-dimensional torus $\Gm$.
The same is true for $\Aut(X)^0$ (see \eqref{eq:exact-sequence}).
Hence
$\Aut(X)^0$ is a the semi-direct product 
 $\Ga \rtimes \Gm$ (for suitable action of $\Gm$ on $\Ga$).
We obtain the case~\ref{lemma:group:GaGm}.
\end{proof}

\begin{lemma}
\label{lemma:aut}
Let $X\subset \PP(1^3,2,3)$ be an embedding as a hypersurface of degree $6$.
Then any automorphism $g\in \Aut(X)^0$ is induced by an automorphism of $\PP(1^3,2,3)$.
\end{lemma}

\begin{proof}
Note that the Picard group of $X$ is cyclic and generated by the class of $A_X$.
Hence any line bundle on $X$ is $\Aut(X)$-invariant.
Note that in all cases of Lemma~\ref{lemma:group}
we have $\Pic( \Aut(X)^0 )=0$ since in these cases $\Aut(X)^0$ can be realized as an open subset of an 
affine space. Therefore, any
line bundle on $X$ admits a $\Aut(X)^0$-linearization (see e.g. \cite[Theorem~7.2]{Dolgachev:L-invariant}).
This applies to  $\OOO_X(nA_X)$ for any $n$.
Therefore, 
the $\Aut(X)^0$-action on~$X$ naturally lifts to an $\Aut(X)^0$-action on the graded algebra 
\begin{equation*}
\mathrm{R}(X,A_X)= \bigoplus_{n\ge 0} H^0(X, \OOO_X(nA_X)).
\end{equation*}
Since $A_X$ is ample, we have $X=\Proj \mathrm{R}(X,A_X)$.
It is known that $\mathrm{R}(X,A_X)$ is generated by its elements $x_1,x_2,x_3,y,z$
of degrees $\deg x_i=1$, $\deg y=2$, $\deg z=3$, and 
there is a unique quasi-homogeneous relation $\phi(x_1,x_2,x_3,y,z)=0$ of degree $6$
(see e.g. \cite{Shin1989} or \cite[\S~3.4]{P:Fano25e}).
This means that there is an exact sequence
\begin{equation*}
0 \longrightarrow (\phi) \longrightarrow \Bbbk[x_1,x_2,x_3,y,z] 
\overset{\alpha}{\longrightarrow} \mathrm{R}(X,A_X) \longrightarrow 0,
\end{equation*}
where $\Bbbk[x_1,x_2,x_3,y,z]$ is regarded as a graded ring with $\deg x_i=1$, $\deg y=2$, $\deg z=3$. 
The surjection $\alpha$ induces an embedding $X \hookrightarrow \PP(1^3,2,3)$,
moreover, any such an embedding is obtained in such a  way.
Since the degree of $\phi$ is strictly less that 
the degrees of variables, the $\Aut(X)^0$-action lifts from 
$\mathrm{R}(X,A_X)$ to $\Bbbk[x_1,x_2,x_3,y,z]$ so that
the ideal $(\phi)$ is invariant. Therefore, the $\Aut(X)^0$-action on $\Bbbk[x_1,x_2,x_3,y,z]$
induces a $\Aut(X)^0$-action on $\Proj\Bbbk[x_1,x_2,x_3,y,z]= \PP(1^3,2,3)$ so that
the embedding $X \hookrightarrow \PP(1^3,2,3)$ considered above is $\Aut(X)^0$-equivariant.
\end{proof}

\begin{setup}
\label{setup:GaGm}
Let~$X$ be a double Veronese cone such that the 
identity component $\Aut(X)^0$ of $\Aut(X)$ is not trivial.
By Lemma~\ref{lemma:group} the group $\Aut(X)^0$ is isomorphic either $\Gm$, $\Ga$ or $\Ga \rtimes \Gm$.
Recall that there is an embedding $X\subset \PP(1^3,2,3)$
whose image is a hypersurface of degree~$6$ 
and this embedding
is $\Aut(X)^0$-equivariant (see Lemma~\ref{lemma:aut}).
Moreover, by changing coordinates in $\PP(1^3,2,3)$ we may assume that 
$X$ is given by the equation \eqref{eq:eq}.
If furthermore the group $\Aut(X)^0$ is an algebraic torus, then it is reductive and we may assume that 
the coordinates $x_1,x_2,x_3, y,z$ are semi-invariants.
\end{setup}

Now, we investigate additive group actions on double Veronese cones.

\begin{notation}
For $\lambda\in \PP^1=\Bbbk\cup \{\infty\}$, consider the quadratic form
\begin{equation}
\label{eq:psi-lambda}
\psi_\lambda:= 
\begin{cases}
x_2^2-2 x_1x_3+\lambda x_3^2&\text{if $\lambda\in \Bbbk$,} \\
x_3^2& \text{if $\lambda=\infty$.}
\end{cases}
\end{equation} 
It defines a pencil of conics
$\{ D_\lambda\}_{\lambda\in \PP^1}$,
where
\begin{equation}
\label{eq:Dlambda}
D_\lambda=\{\psi_\lambda=0\}.
\end{equation}
Clearly, $D_{\infty}$ is the double line $L:=\{x_3=0\}$ and
for $\lambda\neq \infty$ the conic $D_\lambda$ is smooth and the line $L$
is tangent to $D_\lambda$ at $Q:=(1,0,0)$ (see Fig.~\ref{fig:Pencil}).
\begin{figure}[htbp]
\centering
\begin{tikzpicture}
\fill (2.1,-0.) circle (2pt) node[right, xshift=5] {$Q$};
\draw[black, ultra thick] (2.1,2) -- (2.1,-2) node[right, xshift=5] {$L$};
\draw[thick,-latex] (0, 0) ellipse [x radius =5em, y radius = 2em] ;
\draw[thick,-latex] (-0.41,-0) ellipse [x radius =6em, y radius = 3em]; 
\draw[thick,-latex] (-0.8,-0) ellipse [x radius =7em, y radius = 4em] node[above, xshift=-5] {$D_\lambda$};
\end{tikzpicture} 
\caption{$\Ga$-invariant pencil of conics.}
\label{fig:Pencil}
\end{figure}
If we define a $\Ga$-action on $\PP^2$ via 
\begin{equation}
\label{eq:Ga-action:B}
\textstyle
(x_1,x_2,x_3) \longmapsto \left( x_1+ x_2 t +\frac 12 x_3 t^2,\ x_2+ x_3 t,\ x_3\right),
\end{equation}
then the pencil $\{ D_\lambda\}$ is $\Ga$-invariant.
More precisely, the conics $D_\lambda$ are closures of $\Ga$-orbits.
\end{notation}

\begin{proposition}
\label{prop:Ga}
Suppose that a double Veronese cone~$X$ admits an effective action of the additive group $\Ga$. 
Then in a suitable coordinate system in $\PP(1^3,2,3)$ 
the variety $X\subset \PP(1^3,2,3)$ can be given by the equation
\begin{equation}
\label{eq:eq:Ga:eq}
z^2+y^3+\varepsilon y \psi_{\lambda'_1} \psi_{\lambda'_2} +\psi_{\lambda_1} \psi_{\lambda_2} \psi_{\lambda_3}=0,
\end{equation}
where $\varepsilon\in \Bbbk$ is a constant and $\psi_{\lambda_i}$, $\psi_{\lambda'_j}$ are given by \eqref{eq:psi-lambda}, and 
$\Ga$ acts on~$X$ via
\begin{equation}
\label{eq:Ga-action:X}
\textstyle
(x_1,x_2,x_3,y,z) \longmapsto \left( x_1+ x_2 t +\frac 12 x_3 t^2,\ x_2+ x_3 t,\ x_3,\,y,\,z\right).
\end{equation}
where $\varepsilon\in \Bbbk$ is a constant and the quadratic forms $\psi_{\lambda_i}$, $\psi_{\lambda'_j}$ are given by \eqref{eq:psi-lambda}.
The singular locus of~$X$ consists of a single point $P:=(1,0,0,0,0)$, which is of type \type{cD}{} or \type{cE}{}.
\end{proposition}

\begin{proof}
Recall that the induced $\Ga$-action  on $B=\PP^2$ is faithful.
Any one-dimensional additive subgroup of $\PGL_3(\Bbbk)$ is
the image of the homomorphism
\begin{equation*}
\Ga \longrightarrow \PGL_3(\Bbbk),\qquad t \longmapsto \exp(tN),
\end{equation*}
where 
$N$ is a nilpotent matrix. 
Up to conjugacy there are two cases
\begin{equation*}
N=\begin{pmatrix}
0& 0 & 0
\\
0& 0 & 1
\\ 
0& 0 & 0
\end{pmatrix}
\quad\text{or} \quad 
N=\begin{pmatrix}
0& 1 & 0
\\
0& 0 & 1
\\ 
0& 0 & 0
\end{pmatrix}
\end{equation*}
In the former case the action has the form
\begin{equation*}
(x_1,x_2,x_3) \longmapsto (x_1,\ x_2+ x_3 t,\ x_3).
\end{equation*}
Then any $\Ga$-invariant curve is a line passing through $Q:=(0,1,0)$. 
In particular, all the components of $\Lambda_4$ and $\Lambda_6$ 
are lines passing through $Q$. 
This contradicts Corollary~\ref{cor:w-blowup}.

In the latter case $\Ga$ acts on $\PP^2$ as in \eqref{eq:Ga-action:B}.
Then any $\Ga$-invariant curve is a member of the pencil $\{ D_\lambda\}$ (see~\eqref{eq:Dlambda}).
Therefore, the equations of $\Lambda_4$ and $\Lambda_6$ must be products 
of quadratic forms $\psi_\lambda$ (see~\eqref{eq:psi-lambda}) and so they are given by
\begin{equation*}
\begin{array}{lll}
\Lambda_4&=& \{\psi_{\lambda'_1} \psi_{\lambda'_2}=0\} \quad \text{or}\quad \Lambda_4=\PP^2,
\\
\Lambda_6&=& \{\psi_{\lambda_1} \psi_{\lambda_2} \psi_{\lambda_3}=0\}.
\end{array} 
\end{equation*} 
Then in some coordinate system in $\PP(1^3,2,3)$ the equation of~$X$ can be written in the form~\eqref{eq:eq:Ga:eq} and the group $\Ga$ acts on~$X$ via \eqref{eq:Ga-action:X}.
On the other hand, this $\Ga$-action is unique by Lemma~\ref{lemma:group}.
Hence~\eqref{eq:Ga-action:X} is the only choice.

Since $Q=(1,0,0)$ is the only $\Ga$-fixed point,
any singularity of~$X$ is contained in the fiber $\tau^{-1}(Q)$.
Since $\tau^{-1}(Q)$ is an irreducible curve of genus $1$, we have $\Sing(X)=\Sing(\tau^{-1}(Q))$ is a single point.
Finally, this point is not of type \type{cA}{} by Lemma~\ref{lemma:singularities:6}.
\end{proof} 

\begin{example}
In the notation of Proposition~\xref{prop:Ga} assume that $\varepsilon=0$, the constants $\lambda_1,\lambda_2,\lambda_3$ are distinct, and 
$\lambda_i\neq \infty$ for $i=1,2,3$, then the singularity $P\in X$ is terminal of type \type{cD}{4}.
\end{example}

\begin{corollary}
Let~$X$ be a double Veronese cone. Then $\dim \Aut(X)\le 1$.
Therefore, in the case where $\Aut(X)$ is infinite, the identity component $\Aut(X)^0\subset \Aut(X)$ is isomorphic either to $\Ga$ or $\Gm$.
\end{corollary}
\begin{proof}
Assume that $\dim \Aut(X)\ge 2$. Then $\Aut(X)^0\simeq \Ga \rtimes \Gm$
(see Lemma~\ref{lemma:group}).
Let $G_{\mathrm{u}}\subset \Aut(X)^0$ be the unipotent radical.
Thus $G_{\mathrm{u}}\simeq \Ga$ and $\Aut(X)^0$ contains a one-dimensional torus $G_{\mathrm{t}}\simeq \Gm$.
In this case $\Aut(X)^0$ preserves the pencil $\{D_\lambda\}$ (see Fig.~\ref {fig:Pencil}). 
If the induced $G_{\mathrm{t}}$-action  on this pencil is trivial, then $G_{\mathrm{t}}$ has two distinct fixed points $Q$ and $Q_\lambda$
on each member $D_\lambda$, $\lambda\neq \infty$. 
The tangent line $T_{Q_\lambda, D_\lambda}$ to $D_\lambda$ at $Q_\lambda$ is $G_{\mathrm{t}}$-invariant and meets other (invariant)
members of the pencil. This implies that the $G_{\mathrm{t}}$-action on $T_{Q_\lambda, D_\lambda}$ (and also on $B$) is trivial, a contradiction.

Thus $G_{\mathrm{t}}$ has exactly two invariant members in the pencil $\{D_\lambda\}$.
One of them must be $D_\infty$ and another one we denote by $D_{\lambda_0}$.
Then $\Supp(\Lambda_4),\,\Supp(\Lambda_6)\subset D_\infty\cup D_{\lambda_0}$. 
By Corollary~\ref{cor:sing:comp} we have $\Lambda_4\neq \varnothing$.
Thus, 
\begin{equation*}
\phi_4=c_1\psi_{\infty}^{k_1}\psi_{\lambda_0}^{k_0},
\qquad 
\phi_6=c_2\psi_{\infty}^{l_1}\psi_{\lambda_0}^{l_0},
\qquad 
4\phi_4^3+27\phi_6^2=4c_1^3\psi_{\infty}^{3k_1}\psi_{\lambda_0}^{3k_0}+27c_2^2\psi_{\infty}^{2l_1}\psi_{\lambda_0}^{2l_0},
\end{equation*}
where $k_1+k_0=2$ and $l_1+l_0=3$.
Since $\Supp(\Di)\subset D_\infty\cup D_{\lambda_0}$, we obtain $3k_1=2l_1$ and $3k_0=2l_0$. 
Then either $k_1\ge 1$ and $l_1\ge 3$ or $k_0\ge 1$ and $l_0\ge 3$.
This contradicts Lemma~\ref{lemma:sing0}.
\end{proof}

Finally, we investigate actions of the one‑dimensional torus $\Gm$ on double Veronese cones. 
\begin{notation}
We say that a semi‑invariant polynomial $\phi$ of the torus $\Gm$ \textit{has weight $w$} if $\Gm$ acts on $\phi$ via $\phi \longmapsto t^w\phi$. In this situation we write $\wt(\phi)=w$.
\end{notation}

\begin{proposition}
\label{prop:torus}
Let~$X$ be a double Veronese cone admitting an effective $\Gm$-action. Then the equation of~$X$ and the action are described in Table~\xref{table:Gm}. The singular locus of $X$ consists of at most
three points.

Conversely, a general member of each family
defined by $\phi_4$ and $\phi_6$ as in Table~\xref{table:Gm} is a double Veronese cone admitting an effective $\Gm$-action.

\begin{table}[h]
\centering
\begin{tabularx}{0.9999\textwidth}{|l|l|c|X|c|l|}
\hline
&\multicolumn{1}{|c|}{$\wt$}
&\multicolumn{2}{|c|}{$\phi_6$}
&\multicolumn{2}{|c|}{$\phi_4$}
\\\hhline{~~----}
&\multicolumn{1}{|c|}{$(x_1,x_2,x_3,y,z)$}
&\multicolumn{1}{|c|}{$\wt$}
&\multicolumn{1}{|c|}{\rm equation}
&\multicolumn{1}{|c|}{$\wt$}
&\multicolumn{1}{|c|}{\rm equation}
\\\hline
\nr 
\label{0,0,1-}
& $(0,0,6, 2,3)$ & $6$ & $x_3\psi(x_1,x_2)$, $\deg \psi=5$ & - & $0$ 
\\
\nr
\label{0,1,2-5}
& $(0,6,12,10,15)$ & $30$ & $(a_1x_1^2x_3^2 +a_2 x_1x_2^2x_3 + a_3 x_2^4)x_1x_2 $ & - & $0$
\\
\nr
\label{0,1,2-6} 
& $(0,1,2,2,3)$ & $6$ &
$ a_1x_1^3x_3^3+ a_2x_1^2 x_2^2x_3^2+ a_3x_1x_2^4x_3+a_{4}x_2^6$ 
& $4$ & $b_1x_1^2x_3^2+b_2 x_1x_2^2x_3+b_3x_2^4$
\\
\nr
\label{0,1,3-6}
& $(0,1,3,2,3)$ & $6$ & $a_1x_1^4x_3^2+ a_2x_1^2x_2^3x_3+ a_3x_2^6$ & $4$ & $(b_1x_1^2 x_3+b_2x_2^3)x_2$
\\
\nr
\label{0,1,3-7} 
& $(0,6,18,14,21)$ & $42$ & $(a_1x_1^2x_3+ a_2x_2^3)x_1x_2x_3$ & - & $0$
\\
\nr
\label{0,1,4-5} 
& $(0,6,24,10,15)$ & $30$ & $(a_1x_1^3 x_3 + a_2x_2^4)x_1x_2 $ & - & $0$ 
\\
\nr
\label{0,1,4-6} 
& $(0,1,4,2,3)$ & $6$ & $(a_1x_1^3x_3+ a_2x_2^4)x_2^2$ & $4$ & $b_1x_1^3x_3+b_2x_2^4$ 
\\
\nr
\label{0,1,4-8} 
& $(0,3,12,8,12)$ & $24$ & $(a_1x_1^3x_3 + a_2 x_2^4)x_1x_3 $ & - & $0$
\\
\nr
\label{0,1,4-9} 
& $(0,2,8,6,9)$ & $18$ & $(a_1x_1^3 x_3^2+ a_2x_2^4x_3) x_2$ & $12$ & $bx_1x_2^2x_3$ 
\\
\nr
\label{0,1,5-5} 
& $(0,6,30,10,15)$ & $30$ & $(a_1x_1^4x_3 + a_2x_2^5)x_1 $ & - & $0$
\\
\nr
\label{0,1,5-6}
& $(0,1,5,2,3)$ & $6$ & $(a_1x_1^4x_3+ a_2x_2^5)x_2$ & $4$ & $bx_2^4$
\\
\nr
\label{0,1,5-10} 
& $(0,3,15,10,15)$ & $30$ & $(a_1x_1^4 x_3 + a_2x_2^5)x_3$ & - & $0$
\\
\nr
\label{0,2,5-10} 
& $(0,6,15,10,15)$ & $30$ & $(a_1x_1^3x_3^2 + a_2x_2^5)x_1$ & - & $0$
\\
\nr
\label{0,2,5-12} 
& $(0,2,5,4,6)$ & $12$ & $(a_1x_1^3x_3^2+ a_2x_2^5)x_2$ & $8$ & $bx_2^4$
\\
\nr
\label{0,2,5-15} 
& $(0,4,10,10,15)$ & $30$ & $(a_1x_1^3x_3^2+ a_2x_2^5)x_3 $ & $20$ & $bx_1^2x_3^2$ 
\\
\nr
\label{0,1,6-6} 
& $(0,1,6,2,3)$ & $6$ & $a_1x_1^5x_3+ a_2x_2^6$ & $4$ & $bx_2^4$
\\
\nr
\label{0,1,8-12} 
& $(0,1,8,4,6)$ & $12$ & $ax_1x_2^4x_3 $ & $8$ & $bx_1^3x_3 $
\\
\nr
\label{0,1,10-15} 
& $(0,2,20,10,15)$ & $30$ & $ax_2^5x_3 $ & $20$ & $bx_1^3x_3 $
\\
\nr
\label{0,3,10-15}
& $(0,6,20,10,15)$ & $30$ & $ax_1x_2^5$ & $20$ & $bx_1^3x_3 $
\\\hline
\end{tabularx}
\caption{Double Veronese cones admitting effective $\Gm$-actions.}
\label{table:Gm}
\end{table}
\end{proposition}
\begin{example}
The double Veronese cone given by \eqref{eq:E8} is as in \ref{0,0,1-}
of Table~\ref{table:Gm}.
It has a unique singular point that is of type \type{cE}{8}.
Recall that it is $\QQ$-factorial by Example~\ref{ex:Q-fact}.
\end{example}

\begin{proof}
By the arguments in~\ref{setup:GaGm} we may assume that our double Veronese cone
$X$ is given by the equation \eqref{eq:eq} so that 
the coordinates $x_1,x_2,x_3, y,z$ are semi-invariants.
Then the polynomials $\phi_6$ and $\phi_4$ are also semi-invariants with 
\begin{equation}
\label{wt:6}
\wt(\phi_6)=3\wt(y)=2\wt(z).
\end{equation} 
It follows that the $\Gm$-action on~$X$ is determined by the one on $x_1,x_2,x_3$. 
Thus for our purposes it is sufficient to describe the induced $\Gm$-actions on $B=\PP^2$.
Note also that in the case $\phi_4\neq 0$ we have
\begin{equation}
\label{eq:wt:4-6}
2\wt(\phi_6)=3\wt(\phi_4).
\end{equation} 

First, assume that the fixed point locus of $\Gm$ on $B=\PP^2$ is one-dimensional. 
Then its one-dimensional component must be a line and in suitable coordinates we may assume that $\Gm$ acts on $B$ 
so that $x_1$ and $x_2$ are invariant.
In this case,
\begin{equation*}
\phi_4=x_3^k\psi_{4-k} (x_1,x_2),\qquad \phi_6=x_3^l\psi_{6-l} (x_1,x_2),
\end{equation*}
where $\psi_{4-k}$ and $\psi_{6-l}$ homogeneous forms of degree~$4-k$ and $6-l$, respectively.
If $\phi_4=0$, then $l=1$ by Lemma~\ref{lemma:sing0} and we obtain the case \ref{0,0,1-}.
Thus we may assume that $\phi_4\neq 0$.
In this case $2l=2\wt (x_3^l\psi_{6-l}) =3\wt (x_3^k\psi_{4-k})=3k$ by \eqref{eq:wt:4-6}
and so $k=l=0$ again by Lemma~\ref{lemma:sing0}. 
Therefore $\Lambda_4$ and $\Lambda_{6}$ are unions of lines passing through 
the point $(0,0,1)$. This contradicts Corollary~\ref{cor:w-blowup}.

Now, consider the case where the fixed point locus is zero-dimensional. Then it consists of three 
distinct points $Q_1,Q_2,Q_3\in \PP^2$ in general position.
We may assume that
the $\Gm$-action  on $B=\PP^2$ is given by
\begin{equation*}
(x_1,x_2,x_3) \longmapsto (x_1,t^m x_2,t^nx_3),\qquad 0<m\le n/2.
\end{equation*}
Let $L_i$ be the line given by $x_i=0$. Put
\begin{equation*}
d:=\gcd(n,m),\quad n^{\sharp}:=n/d,\quad m^{\sharp}:=m/d. 
\end{equation*} 
Denote $w_6:=\wt(\phi_6)/d$ and $w_4:=\wt(\phi_4)/d$.
We can write
\begin{equation*}
\phi_6=\sum _{i=1}^N a_ix_1^{k_{1,i}}x_2^{k_{2,i}}x_3^{k_{3,i}},\qquad
\phi_4=\sum _{j=1}^{N'} b_jx_1^{l_{1,j}}x_2^{l_{2,j}}x_3^{l_{3,j}},
\end{equation*}
where 
\begin{align}
\label{eq:k}
k_{1,i}+k_{2,i}+k_{3,i}&=6,\quad
m^{\sharp}k_{2,i}+n^{\sharp}k_{3,i}=w_6,
\\
\label{eq:l}
l_{1,j}+l_{2,j}+l_{3,j}&=4,\quad
m^{\sharp}l_{2,j}+n^{\sharp}l_{3,j}=w_4.
\end{align}
Admitting $a_i=0$ or $b_j=0$, we may assume that $\phi_6$ (resp. $\phi_4$) contains all the terms satisfying~\eqref{eq:k} (resp.~\eqref{eq:l}).

First consider the case where $N>1$, i.e. the system of equations~\eqref{eq:k} has at least two solutions
$(k_{1,1},k_{2,1},k_{3,1})\neq (k_{1,2},k_{2,2},k_{3,2})$.
We may assume that $k_{3,1}\ge k_{3,2}$. Then by~\eqref{eq:k}
\begin{equation*}
m^{\sharp}(k_{2,2}-k_{2,1})=n^{\sharp}(k_{3,1}-k_{3,2}), 
\end{equation*}
where both sides are positive. 
Since $\gcd(n^{\sharp},m^{\sharp})=1$, we have $k_{2,2}\equiv k_{2,1}\mod n^{\sharp}$.
In particular, $n^{\sharp}\le k_{2,2}-k_{2,1}\le 6$ and $m^{\sharp}\le 2$.
If $w_6 \equiv 0\mod 3$, we obtain the cases
\ref{0,1,2-6},
\ref{0,1,3-6},
\ref{0,1,4-6},
\ref{0,1,4-9},
\ref{0,1,5-6},
\ref{0,2,5-12},
\ref{0,2,5-15},
\ref{0,1,6-6}
as well as the following ones:
\par
\begin{tabularx}{0.9\textwidth}{llllXXl}
& $(m^{\sharp},n^{\sharp})$ & $w_6$ & $w_4$ & $\phi_6$ & $\phi_4$& $\gcd(\phi_6,\, \phi_4)$
\\
\nr & $(1,2)$ & $9$ & $6$ & $(a_1x_1x_3 + a_2x_2^2) x_2x_3^3$ & $(b_1x_1x_3 + b_2x_2^2)x_3^2$ & $x_3^2$
\\
\nr
\label{0,1,2-3}
& $(1,2)$ & $3$ & $2$ & $(a_1x_1x_3+ a_2x_2^2)x_1^3x_2 $ & $(b_1x_1x_3+b_2 x_2^2)x_1^2 $ & $x_1^2$
\\\nr
\label{0,1,3-3}
& $(1,3)$ & $3$ & $2$ & $(a_1x_1^2x_3+ a_2x_2^3)x_1^3 $ & $b_1x_1^2 x_2^2 $ & $x_1^2$ 
\\\nr
\label{0,1,3-9}
& $(1,3)$ & $9$ & $6$ & $(a_1x_1^3x_3+ a_2x_1x_2^3)x_3^2$ & $(b_1x_1^2x_3+b_2x_2^3)x_3 $ & $x_3$
\\\nr
\label{0,1,3-12}
& $(1,3)$ & $12$ & $8$ & $(a_1x_1^2x_3+ a_2x_2^3)x_3^3 $ & $bx_2^2x_3^2$& $x_3^2$
\\\nr
\label{0,1,4-12}
& $(1,4)$ & $12$ & $8$ & $(a_1x_1^3x_3+ a_2x_2^4)x_3^2$ & $bx_1^2x_3^2$& $x_3^2$
\end{tabularx}
\par
But in the cases 
\ref{0,1,2-3},\ref{0,1,3-3},
\ref{0,1,3-12}, and \ref{0,1,4-12} the polynomials $\phi_6$
and $\phi_4$ have a common multiple divisor,
and in the case \ref{0,1,3-9} they have a common divisor $x_3$ 
that is a multiple for $\phi_6$. This contradicts 
Lemma~\ref{lemma:sing0}.

Thus we may assume that $w_6 \not\equiv 0\mod 3$.
Then 
$\phi_4=0$ by \eqref{eq:wt:4-6}.
Then the polynomial $\phi_6$ has no multiple factors by Lemma~\ref{lemma:sing0}. 
In particular, $\phi_6$ is not divisible by $x_i^2$ for $i=1,2,3$.
We obtain the cases 
\ref{0,1,2-5},
\ref{0,1,3-7},
\ref{0,1,4-5},
\ref{0,1,5-5},
\ref{0,1,5-10},
\ref{0,2,5-10}, as well as the following one:
\begin{itemize}

\item 
$m^{\sharp}=1$, $n^{\sharp}=2$, $\phi_6=(a_1x_1^2x_3^2 + a_2x_1x_2^2x_3 + a_3x_2^4)x_2x_3$.
\end{itemize}
But the latter case can be reduced to 
\ref{0,1,2-5} by switching $x_1$ and $x_3$.

Now 
consider the case where $N=1$, i.e. the system of equations~\eqref{eq:k} has exactly one solution
$(k_{1,1},k_{2,1},k_{3,1})=(k_1,k_2,k_3)$. In other words, 
$\Supp(\Lambda_6)\subset L_1\cup L_2\cup L_3$. 
Then $\Lambda_6$ has a multiple component, hence $\Lambda_4\neq\PP^2$ and $2(m^{\sharp}k_2+n^{\sharp}k_3)=2w_6=3w_4\equiv 0 \mod 3$.

If $N'>1$, 
then by~\eqref{eq:l}
\begin{equation*}
m^{\sharp}(l_{2,2}-l_{2,1})=n^{\sharp}(l_{3,1}-l_{3,2}), 
\end{equation*}
where we may assume that both sides are positive.
Since $\gcd(n^{\sharp},m^{\sharp})=1$, we have $l_{2,2}\equiv l_{2,1}\mod n^{\sharp}$.
In particular, $n^{\sharp}\le l_{2,2}-l_{2,1}\le 4$ and $m^{\sharp}=1$. We obtain $w_6=6$ or $9$ 
and $N>1$. This contradicts our assumption. 
Therefore, $N'=1$, that is,
$\Lambda_4\subset L_1\cup L_2\cup L_3$ and $\phi_4=bx_1^{l_1}x_2^{l_2}x_3^{l_3}$,
where $b\in \Bbbk^*$, $l_1+l_2+l_3= 4$, and $2(m^{\sharp}k_2+n^{\sharp}k_3)=3(m^{\sharp}l_2+n^{\sharp}l_3)$.
Then $2k_2-3l_2$ is divisible by $n^{\sharp}$. In particular, $n^{\sharp}\le \max\{2k_2,\, 3l_2\}\le 12$.
Then $l_i=0$ wherever $k_i\ge 2$ by Lemma~\ref{lemma:sing0}.
We obtain the cases \ref{0,1,8-12},
\ref{0,1,10-15},
\ref{0,3,10-15}.
\end{proof}

\begin{proposition}
Let~$X$ be a nodal double Veronese cone. Then $\Aut(X)$ is finite.
\end{proposition}

\begin{proof}
Assume the contrary.
Then there exists a one-dimensional connected subgroup $G\subset \Aut(X)$ so that $G\simeq \Ga$ or $\Gm$.
Note that $\Lambda_4\neq\PP^2$ by Corollary~\ref{cor:nodal}\ref{cor:nodal:0}
and $\Lambda_4$ and $\Lambda_6$ have no common components by Corollary~\ref{cor:nodal}\ref{cor:nodal:CC}.
If $G\simeq \Ga$, then~$X$ has a unique singular point, say~$P$,
and both $\Lambda_4$ and $\Lambda_6$ are singular at $\tau(P)$ by Proposition~\ref{prop:Ga}.
This contradicts Corollary~\ref{cor:nodal}\ref{cor:nodal:1}.

Thus we may assume that $G\simeq \Gm$. We use the notation of the proof of Proposition~\ref{prop:torus}.
If the fixed point locus of $G$ on $B$ has positive dimension, then we are in the case \ref{0,0,1-}
of Proposition~\ref{prop:torus}. But then $\Lambda_4=\PP^2$, a contradiction.

Therefore, the fixed point locus of $G$ on $B$ consists of 
three distinct points $Q_1,Q_2,Q_3\in \PP^2$ in general position.
If $\Lambda_6\subset L_1\cup L_2\cup L_3$, then $\Lambda_6=2(L_1+ L_2+ L_3)$ by Corollary~\ref{cor:nodal}\ref{cor:nodal:Lambda6}.
In this case the variety
$X$ has a singularity over each point in $\Supp(\Lambda_4)\cap\Supp(\Lambda_6)$ 
by \eqref{eq:singularities:3}.
On the other hand,
each point in $\Supp(\Lambda_4)\cap\Supp(\Lambda_6)$ must be fixed by $G$,
hence $\Supp(\Lambda_4)\cap\Supp(\Lambda_6)\subset \{Q_1,Q_2,Q_3\}$.
Then by Lemma~\ref{lemma:singularities:6} we have
\begin{equation*}
24= \Lambda_4\cdot \Lambda_6=\sum_{i=1}^3 (\Lambda_4\cdot \Lambda_6)_{Q_i}\le 6,
\end{equation*}
a contradiction. Therefore, $\Lambda_6\not\subset L_1\cup L_2\cup L_3$.
By Corollary~\ref{cor:nodal}\ref{cor:nodal:Lambda6} we also have $\Lambda_4\not\subset L_1\cup L_2\cup L_3$.
Thus we are left with the cases \ref{0,1,2-6}, 
\ref{0,1,3-6},
\ref{0,1,4-6}.
However, in these cases 
the point $Q_3$ 
is singular for both $\Lambda_4$ and $\Lambda_6$. 
This contradicts Corollary~\ref{cor:nodal}\ref{cor:nodal:1}.
\end{proof}

\section{Rationality and unirationality}
\label{sec:rat}

It is known that a smooth double Veronese cone is not rational.
For general varieties of this type this was proved via degeneration of intermediate Jacobians \cite{Tyurin-middle-Jacobian-e}.
The non‑rationality of arbitrary smooth double Veronese cones was established in \cite{Grinenko:V1MFS}. Moreover \cite{Grinenko:V1MFS} proves that 
there are no any birational conic bundle structures in the smooth case.
In this section we discuss rationality and unirationality questions of \textit{singular} double Veronese cones.
Our main result is Theorem~\ref{thm:rat-unirat}.

\begin{theorem}
\label{thm:rat-unirat}
Let~$X$ be a double Veronese cone.
\begin{enumerate}

\item 
\label{thm:rat-unirat:u} 
If~$X$ is singular, then~$X$ is unirational.

\item 
\label{thm:rat-unirat:r} 
If~$X$ has a singular point $P$ of type \type{cE}{7}, \type{cE}{8} or \type{cD}{n} with $n\ge 6$, then~$X$ is rational.

\item 
\label{thm:rat-unirat:c} 
If~$X$ has a singular point $P$ that is not of type \type{cA}{1}, \type{cA}{2} nor of type \type{cD}{4}, then~$X$ has a \textup(birational\textup) conic bundle structure.
\end{enumerate}
\end{theorem}

\begin{remark}
\begin{enumerate}
 \item
The existence of conic bundle structures on singular Fano threefolds 
was studied in \cite{P:fano-conic}. Thus the part \ref{prop:surf:c} of the theorem
is a small supplement to Theorem~1.2 in \cite{P:fano-conic}. 

\item 
Note that a double Veronese cone~$X$ is rational if $\rr(X)\ge 4$ 
and it has a conic bundle structure $\rr(X)\ge 3$ (see \cite{P:GFano1}).

\item 
Also any double Veronese cone whose automorphism group 
is infinite (see Theorem~\ref{thm:aut}) is also rational.
Indeed, let $G\subset \Aut(X)$ be a one-dimensional connected subgroup.
Thus $G\simeq \Ga$ or $\Gm$. There exists a non-empty open subset $U\subset X$
such that the stabilizer in $G$ of any point $P\in U$ is trivial.
Then the quotient $h: U \to U/G$ is a principal homogeneous space of $G$,
where $U/G$ is a rational surface (because $U/G$ is dominated by a general member 
of $|A_X|$). Moreover, the general fiber $U_\eta$ has a $\Bbbk(U/G)$-point,
hence $h$ has a rational section. Thus $U$ is birationally equivalent to 
the product $(U/G)\times G$, hence it is rational.
\end{enumerate}
\end{remark}

The proof of Theorem~\ref{thm:rat-unirat} is based in the following easy construction
and corresponding facts on singular del Pezzo surfaces over nonclosed fields
see Appendix~\ref{sec:app}.

\begin{proposition}
\label{prop:fib}
In the notation of \eqref{eq:preV1} let~$C$ be a fiber of $\tau$ and let $\PPP\subset |A_X|$ be the pencil of members passing through~$C$.
Let $\hat{\sigma}: \hat{X}\to X$ be the blowup of~$C$.
Then the proper transform $\hat \PPP$ of $\PPP$ on $\hat{X}$ is base point free and defines a contraction $\hat{\varphi}: \hat{X}\to \PP^1$ whose general geometric fiber $\hat{F}$ is a Du Val del 
Pezzo surface of degree~$1$. 
The restriction $\hat{\sigma}_F: \hat{F}\to \hat{\sigma}(\hat{F})$ is an isomorphism and $\hat{\sigma}(F)$ is smooth outside $\Sing(C)$. In particular, $\hat{F}$ has at most one singular point.
Moreover, there is a Zariski open subset $\hat U\subset \hat X$ such that for any $P\in\hat U$ there exists a section of $\hat{\varphi}$ passing through $P$.
\end{proposition}

\begin{proof}
All assertions are almost obvious. Indeed, since~$C$ is the scheme base locus of $\PPP$,
the proper transform $\hat \PPP$ of $\PPP$ under the blowup of~$C$ is base point free, hence it defines 
a morphism $\hat{\varphi}: \hat{X}\to \PP^1$ so that any fiber $\hat{F}$ of $\hat{\varphi}$ is 
a member of $\hat \PPP$. Let $F:=\hat{\sigma}(\hat{F})$. Then $F\in \PPP$ and $\hat{F}$ is the blowup of 
$F$ along the curve~$C$ which is a Cartier divisor on $F$. Therefore, $\hat{F}\simeq F$
is a Gorenstein (possibly non-normal) del Pezzo surface of degree~$1$.
In particular, $\hat{F}$ is Cohen-Macaulay and so $\hat{X}$ is.
On the other hand, by Bertini's theorem a general member $F\in \PPP$ is smooth outside $\Sing(C)$ and a general fiber of $\hat{\varphi}$ has at worst Du Val singularities by 
Proposition~\ref{prop:finite}.
Thus $\hat{X}$ is smooth in codimension one, hence it is normal. We obtain
a contraction $\hat{\varphi}: \hat{X}\to \PP^1$ satisfying the desired properties and the following diagram
\begin{equation}
\label{eq:diag:bl-c}
\vcenter{
\xymatrix @R=1.5em@C=4em{
&\hat{X}\ar[dl]_{\hat \sigma}\ar[dd]_{\hat{\varphi}}
\\
X\ar@{-->}[d]^{\tau}& 
\\
\PP^2\ar@{-->}[r]^{\beta} &\PP^1
}}
\end{equation} 

The proper transform $\hat L\subset\hat X$ of a line $L\subset X$ is a section of 
$\hat{\varphi}$ if and only if $L$ does not meet~$C$.
By Theorem~\ref{thm:lines} and Corollary~\ref{cor:lines} such a line exists 
and moreover, these lines fill a Zariski open subset.
\end{proof} 

\begin{proof}[Proof of Theorem~\xref{thm:rat-unirat}]
In the notation of Proposition~\ref{prop:fib} take the fiber~$C$
passing through a singular point $P\in X$.
Let $\hat X_{\eta}$ be the generic scheme fiber of $\hat \varphi$. This is a Du Val del Pezzo surface over the field $\KK:=\Bbbk(\PP^1)$ and
the set $X_{\eta}(\Bbbk)$ of $\KK$-points is Zariski dense 
in $X_{\eta}$. 
Moreover, $\hat\sigma^{-1}(P)$ is a section of $\hat\varphi$ that induces 
a singular point $P_\eta\in \hat X_{\eta}$. The type of the singularity 
$P_\eta\in \hat X_{\eta}$ is ``not better'' 
than the one of a general geometric hyperplane section of 
$P\in X$.
Now Theorem~\ref{thm:rat-unirat} follows from Proposition~\ref{prop:surf} in the appendix.
\end{proof}

\begin{remark}
Note that the construction \eqref{eq:diag:bl-c} is equivariant with respect to the Bertini involution $\upbeta: X\to X$.
\end{remark}

\appendix 

\section{Du Val del Pezzo surfaces over nonclosed fields}
\label{sec:app}

In this section we work over an arbitrary field $\KK$ of characteristic zero and $\bar \KK$ denotes its algebraic closure.
We discuss birational properties of del Pezzo surfaces of degree~$1$ over $\KK$.

\begin{proposition}
\label{prop:surf}
Let~$S$ be a Du Val del Pezzo surface of degree~$1$. 
\begin{enumerate}

\item 
\label{prop:surf:u} 
If~$S$ is singular and the set $S(\KK)$ of $\KK$-points
is Zariski dense in~$S$, then~$S$ is unirational.

\item 
\label{prop:surf:r} 
If~$S$ has a singular point $P$ of type \type{E}7, \type{E}8, or \type{D}n with $n\ge 6$, then~$S$ is rational.

\item 
\label{prop:surf:c} 
If~$S$ has a singular point $P\in S$ that is not of type \type{A}1, \type{A}2 nor of type \type{D}4, then~$S$ has a \textup(birational\textup) conic bundle structure.
\end{enumerate}
\end{proposition}

Recall the following classical fact.

\begin{theorem}[{\cite{Iskovskikh:79s-e}, \cite{Iskovskikh:Factorization-e}}]
\label{thm:smoothDP}
Let~$S$ be a smooth geometrically rational surface with $S(\KK)\neq \varnothing$. 
\begin{enumerate}

\item 
If $K_S^2\ge 5$, then~$S$ is rational.

\item 
If $K_S^2\ge 4$, then~$S$ has a conic bundle structure.
\end{enumerate}
\end{theorem}

For a Du Val singularity $P\in S$ over $\KK$ define
the number $\mathrm{s}(P)$ as follows:
\begin{equation}
\label{eq:kP}
\mathrm{s}(P)=
\begin{cases}
n & \text{if $P\in S$ is of type \type{E}{n} with $n=7$ or $8$,}
\\
4 & \text{if $P\in S$ is of type \type{E}{6},}
\\
2 & \text{if $P\in S$ is of type \type{D}{4},}
\\
n-1 & \text{if $P\in S$ is of type \type{D}{n} with $n\ge 5$,}
\\
\lceil n/2 \rceil & \text{if $P\in S$ is of type \type{A}{n}.}
\end{cases}
\end{equation}
This is the number of orbits of the action of the automorphism group of the corresponding Dynkin graph on the set of its vertices.
Then we have the following easy fact.

\begin{lemma}
\label{lemma:kP}
Let~$S$ be a surface over $\KK$ and let $P\in S$ be a Du Val $\KK$-point.
Let $\mu: \tilde S\to S$ be the minimal resolution.
Then 
\begin{equation*}
\uprho(\tilde S/S)\ge \mathrm{s}(P).
\end{equation*}
\end{lemma}

\begin{proof}[Proof of Proposition~\xref{prop:surf}]
Consider the minimal resolution $\mu: \tilde{S}\to S$.
Then $-K_{\tilde{S}}$ is nef and big, i.e. $\tilde{S}$ is a smooth weak del Pezzo surface.
Run the MMP:
\begin{equation*}
\psi : \tilde{S}\longrightarrow\hat{S}. 
\end{equation*}
Note that $\tilde{S}(\KK)\neq \varnothing$ since the point
$\Bs |-K_S|$ is smooth and defined over $\KK$.
We end up with a smooth weak del Pezzo surface $\hat{S}$ such that 
one of the following holds:
\renewcommand\labelenumi{\rm \Alph{enumi})}
\renewcommand\theenumi{\rm \Alph{enumi})}
\begin{enumerate}

\item 
\label{A}
$\uprho(\hat{S})=1$ and $\hat{S}$ is a del Pezzo surface;

\item 
\label{B}
$\uprho(\hat{S})=2$ and there is a conic bundle structure $\hat{S}\to \PP^1$.
\end{enumerate}
In particular, $\uprho(\hat{S})\le 2$ and we have
\begin{equation}
\label{eq:DP1:2}
\begin{array}{rcl}
K_{\hat{S}}^2&\ge& K_{\tilde{S}}^2 + \uprho(\tilde{S})-\uprho(\hat{S})
= K_{S}^2+ \uprho(S) + \uprho(\tilde{S}/S) -\uprho(\hat{S})
\\
&\ge& 1+ \uprho(S) +\mathrm{s}(P) -\uprho(\hat{S})\ge 2+\mathrm{s}(P)-\uprho(\hat{S})\ge \mathrm{s}(P). 
\end{array}
\end{equation} 

To prove \ref{prop:surf:u} we note that in the case~\ref{B} the surface $\hat{S}$ is unirational by \cite[Corollary~8]{Kollar-Mella-unirationality}.
Thus we may assume that $\hat{S}$ is a del Pezzo surface. In this case $\psi$ is not an isomorphism, because $-K_{\tilde X}$ is not ample 
and so $K_{\hat{S}}^2>K_{\tilde{S}}^2$. Thus $K_{\hat{S}}^2\ge 2$.
Then \cite[Theorem IV.7.8]{Manin:book:74} (see also \cite{Kollar2002a} and \cite[Corollary~3.3]{Salgado-Testa-Varilly})
and conclude that $\hat{S}$ is unirational. This proves~\ref{prop:surf:u}.
For \ref{prop:surf:r} we just note that $K_{\hat{S}}^2\ge \mathrm{s}(P)\ge 5$
by our assumptions, \eqref{eq:kP}, and \eqref{eq:DP1:2}. Then the rationality of $\hat S$ follows from 
Theorem~\ref{thm:smoothDP}.

To prove \ref{prop:surf:c} assume that~$S$ has no conic bundle structures.
Then the case~\ref{B} does not occur and $3\ge K_{\hat{S}}^2\ge \mathrm{s}(P)+1$ again by Theorem~\ref{thm:smoothDP}.
Thus $K_{\hat{S}}^2=3$, $\mathrm{s}(P)=2$, and $P\in \bar S$ is of type \type{A}{3} or \type{A}{4} by \eqref{eq:kP}. 

Let $C\subset\bar S$ be the element of 
$|-K_{S}|$ passing through $P$ and let $\tilde C\subset \tilde S$ 
be its proper transform. 
Write
\begin{equation*}
\mu^* C=\tilde C+\sum_{i=1}^N a_i E_i, 
\end{equation*}
where $E_i$ are geometric components of the $\mu$-exceptional divisor and $a_i$ are 
positive integers. Then the number $(\sum a_i E_i)^2$ is strictly negative because the 
intersection matrix $\|E_i\cdot E_j\|$ is negative definite.
On the other hand, this number is even because $(\sum a_i E_i)^2=2\p(\sum a_i E_i)-2$.
Hence $(\sum a_i E_i)^2\le -2$.  Thus
\begin{equation}
\label{eq:C2}
\textstyle
\tilde C^2 =C^2+(\sum a_i E_i)^2\le -1,\qquad
\tilde C\cdot K_{\tilde S}= C\cdot K_{S} =-1.
\end{equation} 
Since $-2\le 2\p(\tilde C)-2= \tilde C\cdot K_{\tilde S}+ \tilde C^2$, we obtain $\p(\tilde C)=0$ and $\tilde C^2=-1$,
i.e. $\tilde C$ is a $(-1)$-curve (defined over $\KK$). 
In particular,  we may assume that  $\psi$ contracts $\tilde C$.
Since $K_{\tilde S}^2-K_{\hat  S}^2=2$, the $\psi$ exceptional divisor has 
exactly two components $\tilde C$ and $\tilde D$, where both $\tilde C$ and $\tilde D$
are geometrically irreducible.
If none of $\mu$-exceptional curves $E_i$ are contracted by $\psi$, then $\tilde C$ 
must intersect all these curves (because $\hat S$ does not contain $(-2)$-curves). 
On the other hand, it follows from \eqref{eq:C2} that $\tilde C\cdot \sum a_i E_i=2$, a contradiction.
This implies that $\tilde D=E_j$ for some~$j$. Thus $E_j$ is defined over $\KK$.
If $P\in S$ is of type \type{A}{4}, then all the $\mu$-exceptional curves $E_i$ must be defined over $\KK$ and
$3=\uprho(\hat S)+2=\uprho(\tilde S)\ge 1+4$, a contradiction. Thus $P\in S$ is of type \type{A}{3}.
Note that $\tilde C+\sum a_i E_i=\mu^* C\sim -K_{\tilde S}$, hence $\sum a_i \hat E_i\sim -K_{\hat S}$,
where $\hat E_i:=\psi_* E_i$. Thus the divisor $\sum a_i \hat E_i$ is a hyperplane section of 
the smooth cubic surface $\hat S$ having two components. But then one of these components must be a conic.
\end{proof}

\section{Macaulay2 code}
\label{sect:code}
Here we present the Macaulay2 code used in Example~\ref{ex:PSL27}.

\begin{verbatim}
KK=QQ; R=KK[x_0,x_1,x_2,y];
phi_4=x_1*x_2^3+x_2*x_0^3+x_0*x_1^3; -- the invariant of degree 4
pv= matrix {{x_0,x_1,x_2}}; 
phi_6=(det diff(pv ** transpose pv, phi_4))/54; -- the invariant of degree 6
lambda=3; 
f=y^3+y*lambda*phi_4+phi_6; -- the equation of $X$
f=sub(f,{x_2=>1}); -- affine equation
R1=KK[x_0,x_1,y]; 
f=sub(f,R1);
pv= matrix {{x_0,x_1,y}}; 
polhess =det (diff(pv ** transpose pv, f)); -- Hessian of the equation
degree ((R1/ ideal (diff (x_0,f), diff (x_1,f), diff (y,f),f ))) 
-- the number of singularities of $X$
associatedPrimes(ideal (polhess, diff (x_0,f), diff (x_1,f), diff (y,f),f)) 
-- the singularitie of $X$ that are not ODPs
\end{verbatim}

\bibliography{all,prokho,inv}
\bibliographystyle{alpha}
\end{document}